\documentclass[review]{elsarticle}

\usepackage{lipsum}
\usepackage{amsmath}
\usepackage{amsthm}
\usepackage{amsfonts}
\usepackage{graphicx}
\usepackage{algorithm}
\usepackage{algorithmic}
\usepackage{verbatim}
\usepackage{caption}
\usepackage{subcaption}
\usepackage{float}
\usepackage{booktabs}
\usepackage{array}
\usepackage{lineno,hyperref}
\usepackage[nameinlink,noabbrev]{cleveref}
\usepackage{amsopn}
\usepackage{calc}

\DeclareMathOperator{\diag}{diag}

\newcommand{\ang}[1]{\left\langle {#1} \right\rangle}
\newcommand {\average}[1] {\mbox{$\left\{\!\!\left\{ #1 \right\}\!\!\right\}$}}
\newcommand {\jump}[1] {\mbox{$\left[\!\left[ #1 \right]\!\right]$}}
\newcommand{\set}[2]{\left\{{#1}\,:~{#2}\right\}}

\newtheorem{theorem}{Theorem}[section]

\newtheorem{remark}[theorem]{Remark}

\algsetup{linenosize=\tiny}
\makeatletter

\newcommand{\DESCRIPTION@original@item}{}
\let\DESCRIPTION@original@item\item
\newcommand*{\DESCRIPTION@envir}{DESCRIPTION}
\newlength{\DESCRIPTION@totalleftmargin}
\newlength{\DESCRIPTION@linewidth}
\newcommand{\DESCRIPTION@makelabel}[1]{\llap{#1}}%
\newcommand{\DESCRIPTION@item}[1][]{%
	\setlength{\@totalleftmargin}%
	{\DESCRIPTION@totalleftmargin+\widthof{\textbf{#1 }}-\leftmargin}%
	\setlength{\linewidth}
	{\DESCRIPTION@linewidth-\widthof{\textbf{#1 }}+\leftmargin}%
	\par\parshape \@ne \@totalleftmargin \linewidth
	\DESCRIPTION@original@item[\textbf{#1}]%
}
\newenvironment{DESCRIPTION}
{\list{}{\setlength{\labelwidth}{0cm}%
		\let\makelabel\DESCRIPTION@makelabel}%
	\setlength{\DESCRIPTION@totalleftmargin}{\@totalleftmargin}%
	\setlength{\DESCRIPTION@linewidth}{\linewidth}%
	\renewcommand{\item}{\ifx\@currenvir\DESCRIPTION@envir
		\expandafter\DESCRIPTION@item
		\else
		\expandafter\DESCRIPTION@original@item
		\fi}}
{\endlist}

\begin{document}
	
\begin{frontmatter}
		
	\title{Stochastic Discontinuous Galerkin Methods with Low--Rank Solvers for Convection Diffusion Equations}
						
	\author[mymainaddress]{Pelin \c{C}\.{i}lo\u{g}lu}
	\ead{pciloglu@metu.edu.tr}
		
	\author[mymainaddress]{Hamdullah Y\"ucel\corref{mycorrespondingauthor}}
	\cortext[mycorrespondingauthor]{Corresponding author}
	\ead{yucelh@metu.edu.tr}
		
	\address[mymainaddress]{Institute of Applied Mathematics, Middle East Technical University, 06800 Ankara, Turkey}

	\begin{abstract}
        We investigate numerical behaviour of a convection diffusion equation with random coefficients by approximating  statistical moments of the solution. Stochastic Galerkin approach, turning the original stochastic problem to a system of deterministic convection diffusion equations, is used to handle the stochastic domain in this study, whereas discontinuous Galerkin method is used to discretize spatial domain due to its local mass conservativity. A priori error estimates of the stationary problem  and stability estimate of the unsteady model problem are derived in the energy norm. To address the curse of dimensionality of Stochastic Galerkin method, we take advantage of the low--rank  Krylov subspace methods, which reduce both the storage requirements and the computational complexity by exploiting a Kronecker--product structure of system matrices. The efficiency of the proposed methodology is illustrated by numerical experiments on the benchmark problems.
	\end{abstract}
		
		
\end{frontmatter}
	
	

\section{Introduction}\label{sec:intro}

To simulate complex behaviors of physical systems, ones make predictions and hypotheses about certain outputs of interest with the help of simulation of mathematical models. However, due to the lack of knowledge or inherent variability in the model parameters, such real-problems formulated by mathematical  models generally come with uncertainty concerning computed quantities; see, e.g., \cite{PJRoache_1988}. Therefore, the idea of uncertainty quantification, i.e., quantifying the effects of uncertainty on the result of a computation, has become a powerful tool for modeling physical phenomena in the last few years.

In order to solve PDEs with random coefficients,  there exist three competing methods in the literature: the Monte Carlo method \cite{GSFishman_1996,JSLiu_2008a}, the stochastic collocation method \cite{IBabuska_FNobile_RTempone_2007a,DXiu_JSHesthaven_2005a}, and the stochastic Galerkin method \cite{IBabuska_RTempone_GEZouraris_2004a,RGGhanem_PDSpanos_1991a}. Although the Monte Carlo method is popular for its simplicity,  natural parallelization, and broad applications, it  features slow convergence. For the stochastic collocation methods, the crucial issue  is how to construct the set of collocation points appropriately because the choice of the collocation points determines the efficiency of the method. In contrast to the Monte Carlo approach and the stochastic collocation method, the stochastic Galerkin method is a nonsampling approach, which transforms a PDE with random coefficients into a large system of coupled deterministic PDEs. As in the classic (deterministic) Galerkin method, the idea behind the stochastic Galerkin method is to seek a solution for the model equation such that the residue is orthogonal to the space of polynomials. An important feature of this technique is the separation of the spatial and stochastic variables, which  allows a reuse of established numerical techniques.

In this paper, we mainly focus on the numerical investigation of a convection diffusion equation with random coefficients by using the stochastic Galerkin approach. Corresponding PDE can be considered as a basic model for transport phenomena in random media. For petroleum reservoir simulations or groundwater flow problems, permeability is desperately needed; however, it is hard to accurately measure permeability field in the earth due to the large area of oil reservoir and complicated earth structure. Hence, it is reasonable to model the permeability parameter as a random field, which corresponds to the solution of a convection diffusion equation; see, e.g., \cite{RGhanem_SDham_1998a,CWinter_DTartakovsky_2000}. In the literature, several stochastic finite element methods have been proposed and analysed, see, e.g., \cite{IBabuska_RTempone_GEZouraris_2004a,IBabuska_RTempone_GEZouraris_2005a,MEiermann_OGErnst_EUllmann_2007a,OGErnst_EUllmann_2010,PFrauenfelder_CSchwab_RTodor_2005,HMatthies_AKeese_2005,DXiu_JShen_2009} and the references therein. However, there are a few work on the formulation and analysis of stochastic discontinuous Galerkin method; see, for instance \cite{YCao_KZhang_RZhang_2008,RMYao_LJBo_2007}. To the best of the authors’ knowledge, there exists any study on an analysis of stochastic discontinuous Galerkin methods with convection diffusion equations. With the present paper, we intend to fill this gap. Compared with the discontinuous Galerkin method, the finite difference method is not able to handle complex geometries, the finite volume method is not capable of achieving high--order accuracy, and the standard continuous finite element method lacks the ability of local mass conservation. Moreover, especially for convection dominated problems, DG methods produce stable concretization without the need for stabilization strategies and they allow for different orders of approximation to be used on different elements in a very straightforward manner \cite{DNArnold_FBrezzi_BCockburn_LDMarini_2002a,BRiviere_2008a}.

A major drawback of the stochastic Galerkin methods is the rapid increase of dimensionality, called as the curse of dimensionality. We address this issue by using low--rank  Krylov subspace methods, which reduce both the storage requirements and the computational complexity by exploiting a Kronecker--product structure of system matrices, see, e.g., \cite{JBallani_LGrasedyck_2013,DKressner_CTobler_2011,MStoll_TBreiten_2015}. Similar approaches have been used to solve steady stochastic diffusion equations \cite{SDolgov_BNKhoromskij_ALitvinenko_HGMatthies_2015,KLee_HCElman_2017,HGMatthies_EZander_2012},  unsteady stochastic diffusion equations \cite{PBenner_AOnwunto_MStoll_2015}, and stochastic Navier--Stokes equations \cite{HCElman_TSu_2020,KLee_HCElman_BSousedik_2019}.  In the aforementioned studies, randomness is generally defined in the diffusion parameter however we here consider the randomness both in diffusion or convection parameters.

The rest of the paper is organized as follows: In the next section, we introduce our stationary model problem, that is, a convection diffusion equation with random coefficients, and provide an overview of its discretization, obtained by
Karhunen--Lo\`{e}ve (KL) expansion, stochastic Galerkin method, and discontinuous Galerkin method. In Section~\ref{sec:error}, we derive a priori error estimates for the stationary problem  and stability estimates for the unsteady model problem  in the energy norm. We discuss the implementation of low--rank iterative solvers in Section~\ref{sec:lowrank}. As an extension of the concepts in Section~\ref{sec:model}, we proceed to Section~\ref{sec:unsteady} to introduce and analyze our strategy for the unsteady analogue of the steady--state model. Numerical results are given in Section~\ref{sec:num} to show the efficiency of the proposed approach. Finally, we draw some conclusions and discussions in Section~\ref{sec:conclusions} based on the findings in the paper.


\section{Stationary model problem with random coefficients}\label{sec:model}

Let $\mathcal{D} \subset \mathbb{R}^2$ be a bounded open set with Lipschitz boundary $\partial \mathcal{D}$, and  the triplet $(\Omega, \mathcal{F}, \mathbb{P})$ denotes a complete probability space, where $\Omega$ is a sample space of events, $\mathcal{F} \subset 2^{\Omega}$ denotes a $\sigma$--algebra, and $\mathbb{P}: \mathcal{F} \rightarrow [0,1]$ is the associated probability measure. A generic random field $\eta$ on the probability space $(\Omega, \mathcal{F}, \mathbb{P})$ is denoted by $\eta(x,\omega):\mathcal{D} \times \Omega \rightarrow \mathbb{R}$. For a fixed $x\in \mathcal{D}$, $\eta(x,\cdot)$ is a real--value square integrable random variable, i.e.,
\begin{equation*}
\eta(x, \cdot) \in L^2(\Omega, \mathcal{F}, \mathbb{P}) : = \{X: \Omega \rightarrow \mathbb{R}\, : \; \int_{\Omega} |X(\omega)|^2 \, d\mathbb{P}(\omega) < \infty  \},
\end{equation*}
where $L^2(\Omega)=L^2(\Omega, \mathcal{F}, \mathbb{P})$ is equipped with the norm $\ang{XY} = \left( \int_{\Omega} X\,Y\, d\mathbb{P}(\omega) \right)^{1/2}$. We denote the mean of $\eta(\mathbf{x},\omega)$ at a point $\mathbf{x} \in \mathcal{D}$ by $\overline{\eta}(\mathbf{x}):= \ang{\eta(x,\cdot)}$. Then,  the covariance of $\eta$ at $\mathbf{x},\mathbf{y} \in \mathcal{D}$ is given by
\begin{equation}\label{eqn:cov}
\mathcal{C}_{\eta}(\mathbf{x},\mathbf{y}) := \ang{(\eta(\mathbf{x},\cdot) - \overline{\eta}(\mathbf{x})) (\eta(\mathbf{y},\cdot) - \overline{\eta}(\mathbf{y}))}.
\end{equation}
If we set $\mathbf{x}=\mathbf{y}$ in \eqref{eqn:cov},  the variance $\mathcal{V}_{\eta}$ is obtained at $\mathbf{x} \in \mathcal{D}$. We also note that the standard derivation of $\eta$ is  $\kappa= \sqrt{\mathcal{V}_{\eta}}$.

As a model problem, we first consider  a stationary convection diffusion equation with random coefficients: find a random function $u: \overline{\mathcal{D}} \times \Omega \rightarrow \mathbb{R}$ such that $\mathbb{P}$-almost surely in $\Omega$
\begin{subequations}\label{eqn:m1}
\begin{eqnarray}
-\nabla \cdot ( a(x,\omega) \nabla u(x,\omega) ) + \mathbf{b}(x,\omega)\cdot \nabla u(x,\omega)  & = & f(x)  \quad  \hbox{   in} \;\; \mathcal{D} \times \Omega, \\
u(x,\omega) & = & u_d(x) \quad   \hbox{on} \;\; \partial \mathcal{D} \times \Omega,
\end{eqnarray}
\end{subequations}
where  $a:(\mathcal{D} \times \Omega) \rightarrow \mathbb{R}$ and $\mathbf{b}:(\mathcal{D} \times \Omega) \rightarrow \mathbb{R}^2$ are  random diffusivity and velocity coefficients, respectively, which assumed to have continuous and bounded covariance functions. The functions $f(x) \in L^2(\mathcal{D})$ and $u_d(x)\in L^2(\mathcal{D})$ correspond to the deterministic source term and Dirichlet boundary condition, respectively. To show the regularity of the solution $u$, we need to make the following assumptions:
\begin{DESCRIPTION}
  \item[i)] The diffusivity coefficient $a(x,\omega)$ is $\mathbb{P}$--almost surely uniformly positive, that is, there exist constants $a_{\min}, a_{\max}$ such that $0 < a_{\min} \leq a_{\max} < \infty$, with
\begin{equation}\label{abound}
a_{\min} \leq  a(x,\omega) \leq a_{\max} \qquad \hbox{a. e.} \;\; \hbox{in  } \mathcal{D} \times \Omega.
\end{equation}
\item[ii)] The velocity coefficient $\mathbf{b}$ satisfies $\mathbf{b} \in \big( L^{\infty}(\overline{\mathcal{D}}) \big)^2$ and $\nabla \cdot \mathbf{b}(x,\omega)=0$.
\end{DESCRIPTION}

\noindent Under the assumptions on the coefficients provided above, the well--posedness of the model equation \eqref{eqn:m1} follows from the classical Lax--Milgram lemma; see, e.g., \cite{IBabuska_RTempone_GEZouraris_2004a,GJLord_CEPowell_TShardlow_2014}.

In the following, we introduce the well--known approach Karhunen--L\`{o}eve expansion for the representation of the random coefficients, the solution representation via stochastic Galerkin method and symmetric interior penalty Galerkin method, and the resulting linear system.

\subsection{Finite expansion of random fields}

To solve the model problem \eqref{eqn:m1} numerically, it is needed to reduce the stochastic process into a finite number of mutually uncorrelated, sometimes mutually independent, random variables. Therefore, we assume that the given coefficients $a(x,\omega)$ and  $\mathbf{b}(x,\omega)$ can be approximated by a prescribed finite number of uncorrelated components $\xi_i(\omega), \; i=1, \ldots, N \in \mathbb{N}$, called as finite dimensional noise \cite{IBabuska_RTempone_GEZouraris_2004a,NWiener_1938a}.  Let $\Gamma_i = \xi_i(\Omega) \in \mathbb{R}$ be a bounded interval and $\rho_i \, : \, \Gamma_i \rightarrow [0,1]$ be the probability density functions of the random variables $\xi_i(\omega), \; i=1, \ldots, N \in \mathbb{N}$  with $\omega \in \Omega$. Then, the joint probability density function and the support of such probability density  are  denoted by  $\rho(\xi), \; \xi \in \Gamma$  and  $\Gamma= \prod \limits_{n=1}^{N} \Gamma_n$, respectively.

Following the Karhunen--L\`{o}eve (KL) expansion \cite{KKarhunen_1947,MLoeve_1946}, a random field $\eta(\mathbf{x},\omega): \mathcal{D} \times \Omega \rightarrow \mathbb{R}$ with a continuous covariance function $\mathcal{C}_{\eta}(\mathbf{x},\mathbf{y})$ defined in \eqref{eqn:cov} admits a proper orthogonal decomposition
\begin{equation}\label{eqn:kl}
\eta(\mathbf{x},\omega) = \overline{\eta}(\mathbf{x}) +  \kappa \sum \limits_{k=1}^\infty \sqrt{\lambda_k}\phi_k(\mathbf{x})\xi_k(\omega),
\end{equation}
where $\xi:=\{\xi_1,\xi_2,\ldots\}$ are uncorrelated random  variables. The pair $\{\lambda_k,\phi_k\}$ is a set of  the eigenvalues and eigenfunctions of the corresponding covariance operator $\mathcal{C}_{\eta}$. In order to obtain eigenpairs $\{\lambda_k,\phi_k\}$, ones need to solve the following eigenvalue problem
\[
 \int_{\mathcal{D}}  \mathcal{C}_{\eta}(\mathbf{x},\mathbf{y}) \phi(\mathbf{y}) \; d\mathbf{y} = \lambda_i \phi(\mathbf{x}).
\]
It is noted that as long as the correlation is not zero, the eigenvalues $\{\lambda_k\}$ form a sequence  of nonnegative real numbers decreasing to zero. We approximate $\eta(\mathbf{x},\omega)$ by truncating its KL expansion of the form
\begin{equation}\label{eqn:kltrun}
\eta(\mathbf{x},\omega) \approx \eta_N(\mathbf{x},\omega) := \overline{\eta}(\mathbf{x}) + \kappa \sum \limits_{k=1}^N \sqrt{\lambda_k}\phi_k(\mathbf{x})\xi_k(\omega).
\end{equation}
Here, the choice of the truncated number $N$ is usually based on the speed of decay on the eigenvalues  since
\[
\sum \limits_{i=1}^{\infty}  \lambda_i = \int_{\mathcal{D}} \mathcal{V}_{\eta}(\mathbf{x}) \, d\mathbf{x}.
\]
The truncated KL expansion \eqref{eqn:kltrun} is a finite representation of the random field $\eta(\mathbf{x},\omega)$ in the sense that the mean-square error of approximation is minimized; see, e.g., \cite{IBabuska_PChatzipantelidis_2002a}.

By the assumption based on finite dimensional noise and Doob--Dynkin lemma \cite{BOksendal_2003}, we can replace the probability space $(\Omega, \mathcal{F}, \mathbb{P})$ with $(\Gamma, \mathcal{B}(\Gamma), \rho(\xi)d\xi)$, where $\mathcal{B}(\Gamma)$ denotes  Borel $\sigma$--algebra and $\rho(\xi)d\xi$ is the distribution measure of the vector $\xi$. Then, the solution corresponding to the stochastic PDE \eqref{eqn:m1} admits exactly the same parametrization, that is, $u(\mathbf{x}, \omega) = u(\mathbf{x}, \xi_1(\omega), \xi_2(\omega), \ldots, \xi_N(\omega))$. Hence, we can state the tensor--product space $H^k(\mathcal{D}) \otimes L^2(\Gamma)$, which is endowed with the norm
\[
\|\eta\|_{H^k(\mathcal{D}) \otimes L^2(\Gamma)} := \left( \int_{\Gamma} \|\eta(\cdot,\xi)\|^2_{H^k(\mathcal{D})} \rho(\xi) \, d \xi\right)^{1/2} < \infty.
\]
Further, the following isomorphism relation holds
\[
H^k(\mathcal{D}) \otimes L^2(\Gamma) \simeq L^2(H^k(\mathcal{D});\Gamma) \simeq  H^k(\mathcal{D};L^2(\Gamma)).
\]

\begin{remark}
The covariance functions or eigenpairs can be computed explicitly for some random inputs, such as Gaussian or uniform processes with the exponential covariance function. However, they are usually not known a priori and therefore can be approximated numerically, such as collocation and Galerkin methods, see \cite{GJLord_CEPowell_TShardlow_2014} for more details.
\end{remark}

\subsection{Stochastic Galerkin Method}

The solution of the model problem  \eqref{eqn:m1}, $ u(x,\omega) \in L^2(\Omega, \mathcal{F},\mathbb{P})$, is represented by a generalized polynomial chaos (PC) approximation
\begin{equation}\label{eq:pcekdv}
	u(x,\omega) = \sum_{i=0}^{\infty} u_i(x) \psi_i(\xi(\omega)),
\end{equation}
where $u_i(x)$, the deterministic modes of the expansion,  are given by
\[
u_i(x) = \frac{\ang{u(x,\omega)\psi_i(\xi)}}{\ang{\psi_i^2(\xi)}},
\]
$\xi$ is a finite--dimensional random vector, and $\psi_i$ are multivariate orthogonal polynomials having the following properties:
\[
\ang{\psi_0(\xi)}=1, \qquad \ang{\psi_i(\xi)}=0, \quad i>0, \qquad \ang{\psi_i(\xi)\psi_j(\xi)}=\ang{\psi_i^2(\xi)} \delta_{ij}
\]
with
\[
\ang{\psi_i(\xi)} = \int_{\omega \in \Omega} \psi_i(\xi(\omega)) \, d\mathbb{P}(\omega) = \int_{\xi \in \Gamma} \psi_i(\xi) \rho(\xi)\, d\xi,
\]
where $\Gamma$ and $\rho$ are the support and probability density function of $\xi$, respectively. The orthogonal polynomials, i.e., $\psi_i$,  are chosen according to the type of the distribution of  random input,  for instance, Hermite polynomials and Gaussian random variables, Legendre polynomials and uniform random variables, Laguerre polynomials and gamma random variables \cite{RKoekoek_PALesky_2010}. The probability density functions of random distributions are corresponding to the weight functions of some particular types of orthogonal polynomials.

The Cameron--Martin theorem \cite{RHCameron_WTMartin_1947a} states that the series \eqref{eq:pcekdv} converges in the Hilbert space  $L^2(\Omega, \mathcal{F},\mathbb{P})$. Then, as done in the case of KL expansion \eqref{eqn:kltrun}, we truncate \eqref{eq:pcekdv} as
\begin{equation}
	\label{eq:pcekdvtrun}
	u(x,\omega) \approx u_P(x,\omega)  = \sum_{i=0}^{P-1} u_i(x) \psi_i(\xi(\omega)),
\end{equation}
where the total number of PC basis is determined by the dimension $ N $ of the random vector $ \xi $ and the highest order $ Q $ of the basis polynomials $ \psi_i $
\begin{equation*}	
	P = 1 + \sum \limits_{s=1}^{Q} \frac{1}{s!} \prod \limits_{j=0}^{s-1} (N + j) = \frac{(N+Q)!}{N!Q!}.
\end{equation*}
Then, the corresponding stochastic space is denoted by
\begin{equation}
\mathcal{Y}_n  := \hbox{span}\{\psi_i(\xi): \; i=0,1,\ldots,P-1\} \subset L^2(\Gamma).
\end{equation}
We refer to \cite{OGErnst_EUllmann_2010,CEPowell_HCElman_2009} and references therein for the construction of the stochastic space $\mathcal{Y}_n$.

Next, if we insert  KL expansions \eqref{eqn:kltrun} of the diffusion $a(x,\omega)$ and the convection $\mathbf{b}(x,\omega)$ coefficients, and the solution expression \eqref{eq:pcekdvtrun} into  \eqref{eqn:m1}, we obtain
\begin{eqnarray}\label{CD_SGAA}
&-&\sum_{i=0}^{P-1}\nabla\cdot\Bigg( \Bigg(\overline{a}(x) + \kappa_a \sum_{k=1}^{N}  \sqrt{\lambda_k^{a}} \phi_k^{a}(x) \xi_k\Bigg)\nabla u_i(x)\psi_i \Bigg)  \nonumber \\
&&\qquad + \sum_{i=0}^{P-1} \Bigg(\overline{\mathbf{b}}(x) + \kappa_{\mathbf{b}} \sum_{k=1}^{N}   \sqrt{\lambda_k^{b}} \phi_k^{b}(x) \xi_k\Bigg)\cdot \nabla u_i(x)\psi_i
=f(x).
\end{eqnarray}
By projecting (\ref{CD_SGAA}) onto the space spanned by the PC basis functions, we obtain the following linear system, consisting  of $P$ deterministic convection diffusion equations for $j=0,...,P-1$
\begin{eqnarray}
-\sum_{i=0}^{P-1} \big( \nabla \cdot (a_{ij}\nabla u_i(x)) +  \mathbf{b}_{ij} \cdot \nabla u_i(x) \big) = \ang{\psi_j}f(x),
\end{eqnarray}
where
\begin{eqnarray*}
a_{ij} &=& \overline{a}(x) \ang{ \psi_i^2(\xi)} \delta_{ij} + \kappa_a \sum_{k=1}^N  \sqrt{\lambda_k^{a}} \phi_k^{a}(x) \ang{\xi_k \psi_i(\xi) \psi_j(\xi)}, \\
\mathbf{b}_{ij} &=& \overline{\mathbf{b}}(x) \ang{ \psi_i^2(\xi)} \delta_{ij} + \kappa_{\mathbf{b}}\sum_{k=1}^N  \sqrt{\lambda_k^{\mathbf{b}}} \phi_k^{\mathbf{b}}(x) \ang{\xi_k \psi_i(\xi) \psi_j(\xi)}.
\end{eqnarray*}

\indent Here, the quantity of interest is the statistical moments of the solution $u(x,\omega)$ in \eqref{eqn:m1} rather than  the solution $u(x,\omega)$. Once the modes $u_i, \; i=0,1,\ldots,P-1$, have been computed, the statistical moments and the probability density of the solution  can be easily deduced. For instance, the mean and the variance of the solution are
\begin{equation}
\ang{u(x,\xi)} = u_0(x), \qquad
\mathcal{V}(u(x,\xi)) = \sum \limits_{i=1}^{P-1} u_i^2(x) \ang{\psi_i^2(\xi)},
\end{equation}
respectively.

\subsection{Symmetric interior penalty Galerkin method}

Let $\{ \mathcal{T}_h\}_h$ be a family of shape-regular simplicial triangulations of $\mathcal{D}$. Each mesh $\mathcal{T}_h$ consists of closed triangles such that $\overline{\mathcal{D}} = \bigcup_{K \in \mathcal{T}_h} \overline{K}$ holds. We assume that the mesh is regular in the
following sense: for different triangles
$K_i, K_j \in \mathcal{T}_h$, $i \not= j$, the intersection  $K_i \cap K_j$ is either empty or a vertex or an edge, i.e., hanging nodes are
not allowed. The diameter of an element $K$ and the length of an edge $E$ are denoted by $h_{K}$ and $h_E$, respectively. Further, the maximum value of the element diameter
is denoted by $h=\max \limits_{K \in \mathcal{T}_h} h_{K}$.

We split the  set of all edges $\mathcal{E}_h$ into
the set $\mathcal{E}^0_h$ of interior edges  and the set $\mathcal{E}^{\partial}_h$
of boundary edges so that
$\mathcal{E}_h=\mathcal{E}^{0}_h \cup\mathcal{E}^{\partial}_h$.
Let $\mathbf{n}$ denote the unit outward normal to $\partial \mathcal{D}$.
For a fixed realization $\omega$, the inflow and outflow parts of $\partial \mathcal{D}$ are denoted by $\partial \mathcal{D}^-$ and $\partial \mathcal{D}^+$, respectively,
\[
          \partial \mathcal{D}^- = \set{x \in \partial \mathcal{D}}{ \mathbf{b}(x,\omega) \cdot \mathbf{n}(x) < 0}, \;
          \partial \mathcal{D}^+ = \set{x \in \partial \mathcal{D}}{ \mathbf{b}(x,\omega) \cdot \mathbf{n}(x) \geq 0}.
\]
Similarly, the inflow and outflow boundaries of an element $K$ are defined by
\[
\partial K^-\hspace{-1mm}=\set{x \in \partial K}{\hspace{-2mm}\mathbf{b}(x,\omega) \cdot \mathbf{n}_{K}(x) <0}, \partial K^+\hspace{-1mm}=\set{x \in \partial K}{\hspace{-2mm}\mathbf{b}(x,\omega) \cdot \mathbf{n}_{K}(x) \geq 0},
\]
where $\mathbf{n}_{K}$ is the unit normal vector on the boundary $\partial K$ of an element $K$.

Let the edge $E$ be a common edge for two elements $K$ and $K^e$. For a piecewise continuous scalar function $u$,
there are two traces of $u$ along $E$, denoted by $u|_E$ from inside $K$ and $u^e|_E$ from inside $K^e$.
The jump and average of $u$ across the edge $E$ are defined by:
\begin{align}
\jump{u}&=u|_E\mathbf{n}_{K}+u^e|_E\mathbf{n}_{K^e}, \quad
\average{u}=\frac{1}{2}\big( u|_E+u^e|_E \big).
\end{align}

Similarly, for a piecewise continuous vector field $\nabla u$, the jump and average across an edge $E$ are given by
\begin{align}
           \jump{\nabla u}&=\nabla u|_E \cdot \mathbf{n}_{K}+\nabla u^e|_E \cdot \mathbf{n}_{K^e}, \quad
          \average{\nabla u}=\frac{1}{2}\big(\nabla u|_E+\nabla u^e|_E \big).
\end{align}
For a boundary edge $E \in K \cap \partial \mathcal{D}$, we set $\average{\nabla u}=\nabla u$ and $\jump{u}=u\mathbf{n}$,
where $\mathbf{n}$ is the outward normal unit vector on $\partial \mathcal{D}$.

For an integer $\ell$ and $ K \in \mathcal{T}_h$, let $\mathbb{P}^\ell(K)$ be the set of all polynomials on $K$ of degree at most $\ell$. Then, we define the discrete state and test spaces to be
\begin{align}\label{tspace}
V_h &= \set{u \in L^2(\mathcal{D})}{ u \mid_{K}\in \mathbb{P}^\ell(K) \quad \forall K \in \mathcal{T}_h}.
\end{align}
Note that since discontinuous Galerkin methods impose boundary conditions weakly, the space of discrete states and  test functions are identical.

Following the standard  discontinuous Galerkin structure in \cite{DNArnold_FBrezzi_BCockburn_LDMarini_2002a,BRiviere_2008a},  we define the following (bi)--linear forms for a finite dimensional vector $\xi$:
\begin{eqnarray*}
a_h(u,v,\xi)&=& \sum \limits_{K \in \mathcal{T}_h} \int \limits_{K} a(.,\xi) \nabla u \cdot  \nabla v \, dx
-  \sum \limits_{ E \in \mathcal{E}^{0}_h \cup \mathcal{E}_h^{\partial}} \int \limits_E \average{a(.,\xi)\nabla u }  \jump{v} \, ds\\
&& - \sum \limits_{ E \in \mathcal{E}^{0}_h \cup \mathcal{E}_h^{\partial}} \int \limits_E \average{a(.,\xi) \nabla v }  \jump{u} \; ds
+ \sum \limits_{ E \in \mathcal{E}^{0}_h \cup \mathcal{E}_h^{\partial}} \frac{\sigma }{h_E} \int \limits_E \jump{u} \cdot \jump{v} \, ds   \\
&& + \sum \limits_{K \in \mathcal{T}_h} \int \limits_{K} \mathbf {b}(.,\xi) \cdot \nabla u v \; dx \nonumber  +\hspace{-2.5mm} \sum \limits_{K \in \mathcal{T}_h}\; \int \limits_{\partial K^{-} \backslash \partial \mathcal{D}} \hspace{-4mm}\mathbf {b}(.,\xi) \cdot \mathbf{n}_E (u^e-u)v \, ds\\
&& - \sum \limits_{K \in \mathcal{T}_h} \; \int \limits_{\partial K^{-} \cap \partial \mathcal{D}^{-}} \mathbf {b}(.,\xi) \cdot \mathbf{n}_E u v  \, ds
\end{eqnarray*}
and
\begin{eqnarray*}
l_h(v,\xi)&=&\sum \limits_{K \in \mathcal{T}_h} \int \limits_{K} f v \, dx
	+ \sum \limits_{E \in \mathcal{E}_h^{\partial}}\frac{\sigma}{h_E} \int \limits_E  u_d  \jump{v} \, ds
	-\sum \limits_{E \in \mathcal{E}_h^{\partial}}  \int \limits_E u_d \average{a(.,\xi)\nabla v} \; ds \\
	&& - \sum \limits_{K \in \mathcal{T}_h}  \int \limits_{\partial K^{-} \cap \partial \mathcal{D}^{-}} \mathbf{b}(.,\xi) \cdot \mathbf{n}_E  u_d v  \, ds,
\end{eqnarray*}
where the constant $\sigma>0$  is the interior penalty parameter. It has to be chosen sufficiently large independently of the mesh size to ensure the stability of the DG discretization.

Then, (bi)--linear forms of the stochastic discontinuous Galerkin (SDG)  correspond to
\begin{equation}\label{eq:bilinear}
a_{\xi}(u,v)=\int_{\Gamma}  a_h(u,v,\xi)\rho(\xi) \, d\xi, \quad 	l_{\xi}(v)=\int_{\Gamma}  l_h(v,\xi)\rho(\xi) \, d\xi.
\end{equation}
Now,  we  define the associated energy norm on $\mathcal{D} \times \Gamma$  as
\begin{eqnarray} \label{energynorm}
\lVert u\rVert _{\xi}= \Bigg(\int_{\Gamma}  \lVert u(.,\xi)\rVert_{e}^2 \rho(\xi)\, d\xi \Bigg)^{\frac{1}{2}},
\end{eqnarray}
where $\lVert u(.,\xi)\rVert_{e}$ is the energy norm on $\mathcal{D}$, given as
\begin{eqnarray*}
	\lVert u(.,\xi)\rVert_{e}&=&\Bigg( \sum \limits_{K \in \mathcal{T}_h} \int \limits_{K} a(.,\xi) (\nabla u)^2 \, dx  + \sum \limits_{ E \in \mathcal{E}^{0}_h \cup \mathcal{E}_h^{\partial}} \frac{\sigma}{h_E} \int \limits_E \jump{u}^2 \, ds  \\
&&	+\frac{1}{2}\sum \limits_{ E \in \mathcal{E}_h^{\partial}}\int \limits_E \mathbf {b}(.,\xi)\cdot \mathbf{n}_E u^2 ds
	+ \frac{1}{2}\sum \limits_{ E \in \mathcal{E}^{0}_h }\int \limits_E \mathbf {b}(.,\xi)\cdot \mathbf{n}_E(u^e-u)^2 \, ds\Bigg)^{\frac{1}{2}}.
\end{eqnarray*}
By standard arguments in deterministic case, ones can easily show the coercivity and continuity of $a_{\xi}(\cdot,\cdot)$ for $u, v \in V_h \otimes \mathcal{Y}_n$
\begin{subequations}\label{coer_cont}
\begin{eqnarray}
a_{\xi}(u,u) & \geq & c_{cv} \, \lVert u \rVert _{\xi}^2, \label{coer} \\
a_{\xi}(u,v) & \leq & c_{ct} \, \lVert u\rVert _{\xi}\lVert v\rVert _{\xi}, \label{cont}
\end{eqnarray}
\end{subequations}
where the coercivity constant $c_{cv}$ depends on $a_{\min}$, whereas the continuity constant $c_{ct}$ depends on $a_{\max}$.

Thus, the SDG variational formulation of \eqref{eqn:m1} is as follows: Find $u \in V_h \otimes \mathcal{Y}_n$ such that
\begin{equation}
 a_{\xi}(u,v) = l_{\xi}(v), \qquad  \forall v \in V_h \otimes \mathcal{Y}_n.
\end{equation}

\subsection{Linear System}

After an application of the discretization techniques, one gets the following linear system:
\begin{equation}\label{tensormtrx}
\underbrace{\left( \sum_{i=0}^{N} \mathcal{G}_i \otimes \mathcal{K}_i \right)}_{\mathcal{A}} \, \mathbf{u} = \underbrace{\left( \sum_{i=0}^{N} \mathbf{g}_i \otimes \mathbf{f}_i\right)}_{\mathcal{F}},
\end{equation}
where
$
\mathbf{u} =\left(
u_0, \ldots,u_{P-1}
\right)^T \;\; \hbox{with} \;\;  u_i \in \mathbb{R}^{N_d}, \;\; i=0,1,\ldots,P-1
$
and $N_d$ corresponds to the degree of freedom for the spatial discretization. The stiffness matrices $\mathcal{K}_i \in \mathbb{R}^{N_{d} \times N_d}$  and the right--hand side vectors $\mathbf{f}_i \in \mathbb{R}^{N_d}$  in \eqref{tensormtrx} are given, respectively, by
\begin{eqnarray*}
	\mathcal{K}_0(r,s)\hspace{-3mm}&=&\hspace{-3mm}  \sum \limits_{K \in \mathcal{T}_h} \int \limits_{K} \left(  \overline{a} \, \nabla \varphi_{r} \cdot  \nabla \varphi_{s} + \overline{\mathbf{b}} \cdot \nabla \varphi_{r} \varphi_{s} \right) \, dx \\
&& - \hspace{-2.5mm} \sum \limits_{E  \in \mathcal{E}^{0}_h \cup \mathcal{E}_h^{\partial}} \int \limits_E \big(  \average{\overline{a} \, \nabla \varphi_{r}}  \jump{\varphi_{s}} +  \average{\overline{a} \, \nabla \varphi_{s}}  \jump{\varphi_{r}} \big) \, ds \\
&& + \hspace{-2.5mm} \sum \limits_{ E \in \mathcal{E}^{0}_h \cup \mathcal{E}_h^{\partial}} \frac{\sigma}{h_E} \int \limits_E  \jump{\varphi_{r}} \cdot \jump{\varphi_{s}} \, ds +  \hspace{-2.5mm}  \sum \limits_{K \in \mathcal{T}_h}\; \int \limits_{\partial K^{-} \backslash \partial \mathcal{D}}\hspace{-3mm} \overline{\mathbf{b}} \cdot \mathbf{n}_E (\varphi_{r}^e -\varphi_{r})\varphi_{s} \, ds\\
&&	- \sum \limits_{K  \in \mathcal{T}_h} \; \int \limits_{\partial K^{-} \cap \partial \mathcal{D}^{-}}   \hspace{-3mm}\overline{\mathbf{b}} \cdot \mathbf{n}_E \varphi_{r} \varphi_{s}  \, ds, \\
\mathcal{K}_i(r,s) \hspace{-3mm}&=& \hspace{-3mm} \sum \limits_{K \in \mathcal{T}_h} \int \limits_{K} \left( \Big(\kappa_a \sqrt{\lambda_i^{a}}\phi_{i}^{a}\Big)  \nabla \varphi_{r} \cdot  \nabla \varphi_{s} +  \Big(\kappa_{\mathbf{b}} \sqrt{\lambda_i^{\mathbf{b}}}\phi_{i}^{\mathbf{b}}\Big)  \cdot \nabla \varphi_{r} \varphi_{s} \right) \, dx \\
	&&\hspace{-3mm} -\hspace{-3.5mm}	  \sum \limits_{ E \in \mathcal{E}^{0}_h \cup \mathcal{E}_h^{\partial}} \int \limits_E \left (\average{\Big(\kappa_{a} \sqrt{\lambda_i^{a}}\phi_{i}^{a}\Big)  \nabla \varphi_{r}}  \jump{\varphi_{s}} + \hspace{-1.2mm}	\average{\Big(\kappa_{a} \sqrt{\lambda_i^{a}}\phi_{i}^{a}\Big)  \nabla \varphi_{s}}  \jump{\varphi_{r}} \right) \, ds \\
	&& \hspace{-3mm}+ \sum \limits_{ E \in \mathcal{E}^{0}_h \cup \mathcal{E}_h^{\partial}} \frac{\sigma}{h_E} \int \limits_E  \jump{\varphi_{r}} \cdot \jump{\varphi_{s}} \, ds  \\
&&	\hspace{-3mm} + \sum \limits_{K \in \mathcal{T}_h} \int \limits_{\partial K^{-} \backslash \partial \mathcal{D}} \hspace{-4mm} \Big(\kappa_{\mathbf{b}} \sqrt{\lambda_i^{\mathbf{b}}}\phi_{i}^{\mathbf{b}}\Big)  \cdot \mathbf{n}_E (\varphi_{r}^e-\varphi_{r})\varphi_{s} \, ds \\
&&	\hspace{-3mm} - \sum \limits_{T \in \mathcal{T}_h}  \int \limits_{\partial K^{-} \cap \partial \mathcal{D}^{-}} \hspace{-4mm}\Big(\kappa_{\mathbf{b}} \sqrt{\lambda_i^{\mathbf{b}}}\phi_{i}^{\mathbf{b}}\Big)  \cdot \mathbf{n}_E \varphi_{r} \varphi_{s} \, ds, \\
f_0(s)\hspace{-3mm}&=&\hspace{-3mm}  \sum \limits_{K \in \mathcal{T}_h} \int \limits_{K} f\varphi_{s} \; dx
	+ \sum \limits_{E \in \mathcal{E}_h^{\partial}}\frac{\sigma}{h_E} \int \limits_E  u_d  \jump{\varphi_{s}} \; ds
	-\sum \limits_{E \in \mathcal{E}_h^{\partial}}  \int \limits_E u_d \average{\overline{a} \nabla \varphi_{s}} \; ds \\
	&& \hspace{-3mm}- \sum \limits_{K \in \mathcal{T}_h}  \int \limits_{\partial K^{-} \cap \partial \mathcal{D}^{-}} \overline{\mathbf{b}} \cdot \mathbf{n}_E  u_d \varphi_{s}  \; ds, \\
f_i(s)\hspace{-3mm}&=&\hspace{-3mm} \sum \limits_{E \in \mathcal{E}_h^{\partial}}\frac{\sigma }{h_E} \int \limits_E  u_d \jump{\varphi_{s}} \; ds
	 -\sum \limits_{E \in \mathcal{E}_h^{\partial}}  \int \limits_E u_d \average{\Big(\kappa_{a} \sqrt{\lambda_i^{a}}\phi_{i}^{a}\Big) \nabla \varphi_{s}} \; ds  \;  \\
	\\
	&&\hspace{-3mm} - \sum \limits_{K \in \mathcal{T}_h} \; \int \limits_{\partial K^{-} \cap \partial \mathcal{D}^{-}} \Big(\kappa_{\mathbf{b}} \sqrt{\lambda_i^{\mathbf{b}}}\phi_{i}^{\mathbf{b}}\Big) \cdot \mathbf{n}_E u_d \varphi_{s} \; ds,
\end{eqnarray*}
where $\{\varphi_i(x)\}$ is the set of basis functions for the spatial discretization, i.e., $V_h = \hbox{span} \{\varphi_i(x)\}$.

For $i=0,\ldots, N$  the stochastic matrices $\mathcal{G}_i \in \mathbb{R}^{P \times P}$ in \eqref{tensormtrx} are given by
\begin{equation}\label{stocmtrx}
 \mathcal{G}_0(r,s) = \ang{\psi_r \psi_s}, \qquad  \mathcal{G}_i(r,s)= \ang{\xi_i \psi_r \psi_s},
\end{equation}
whereas the stochastic  vectors $\mathbf{g}_i \in \mathbb{R}^{P}$ in \eqref{tensormtrx} are defined as
\begin{equation}
\mathbf{g}_0(r)= \ang{\psi_r}, \qquad \mathbf{g}_i(r)= \ang{\xi_i \psi_r}.
\end{equation}
In \eqref{stocmtrx} each stochastic basis function $\psi_i(\xi)$ is corresponding to a product of $N$ univariate orthogonal polynomials, i.e.,
$
\psi_i(\xi) = \psi_{i_1}(\xi) \psi_{i_2}(\xi) \ldots \psi_{i_N}(\xi),
$
where the multi--index $i$ is defined by $i=(i_1, i_2, \ldots, i_N)$ with $\sum \limits_{s=1}^N i_s \leq Q$. In this paper, Legendre polynomials are chosen as stochastic basis functions because the underlying random variables have a uniform distribution.

Now, suppose we employ Legendre polynomials in uniform random variables on $(-\sqrt{3}, \sqrt{3})$. Then, recalling the following three--term recurrence for the Legendre polynomials
\begin{eqnarray*}
\psi_{k+1}(x)= \frac{\sqrt{2k+1}\sqrt{2k+3}}{(k+1)\sqrt{3}}x\psi_k(x)-\dfrac{k\sqrt{2k+3}}{(k+1)\sqrt{2k-1}}\psi_{k-1} \; \hbox{with} \; \psi_0=1, \; \psi_{-1} =0,
\end{eqnarray*}
we obtain
\begin{eqnarray*}
\mathcal{G}_0(i,j)&=& \prod_{s=1}^{N} \big\langle {\psi_{i_s}^2(\xi_s) } \big\rangle \delta_{i_sj_s} =  \prod_{s=1}^{N} \delta_{i_sj_s}  =
 \begin{cases}
1, &\mbox{if } i=j, \\
0, & \mbox{otherwise}
\end{cases}
\end{eqnarray*}
and for $k=1:N$
\begin{eqnarray*}
\mathcal{G}_k(i,j)&=& \int_{\Gamma} \xi_k \psi_{i}(\vec{\xi}) \psi_{j}(\vec{\xi}) \rho(\vec{\xi})\;d\vec{\xi}\\
	&=& \int_{-\sqrt{3}}^{\sqrt{3}} \cdots \int_{-\sqrt{3}}^{\sqrt{3}}  \xi_k \psi_{i}(\vec{\xi}) \psi_{j}(\vec{\xi})
	\rho(\vec{\xi})\;d\vec{\xi}\\
	&=&  \Bigg(\prod_{s=1,s\neq k}^{N} \ang{\psi_{i_s}(\xi_s) \psi_{j_s}(\xi_s)}\Bigg) \ang {\xi_k\psi_{i_k}(\xi_k)\psi_{j_k}(\xi_k) } \\
	&=&  \Bigg(\prod_{s=1,s\neq k}^{N} \ang {\psi_{i_s}(\xi_s) \psi_{j_s}(\xi_s)}  \Bigg) \\
   && \;\, \times \left( \dfrac{(i_k+1)\sqrt{3}}{\sqrt{(2i_k+1)(2i_k+3)}}\ang {\psi_{i_k+1}\psi_{j_k} }  +
       \dfrac{i_k\sqrt{3}}{\sqrt{(2i_k+1)(2i_k-1)}}\ang {\psi_{i_k-1}\psi_{j_k} } \right)
\end{eqnarray*}
\begin{eqnarray*}
	&=& \begin{cases}
		\Bigg(\prod \limits_{s=1,s\neq k}^{N} \delta_{i_sj_s}\Bigg) \dfrac{(i_k+1)\sqrt{3}}{\sqrt{(2i_k+1)(2i_k+3)}}, &\mbox{if } i_k+1=j_k, \\
		\Bigg(\prod \limits_{s=1,s\neq k}^{N} \delta_{i_sj_s}\Bigg)\dfrac{i_k\sqrt{3}}{\sqrt{(2i_k+1)(2i_k-1)} }, &\mbox{if } i_k-1=j_k, \\
		0, & \mbox{otherwise}
	\end{cases} \\
	&=& \begin{cases}
	\dfrac{(i_k+1)\sqrt{3}}{\sqrt{(2i_k+1)(2i_k+3)}}, &\mbox{if } i_k+1=j_k \; \hbox{and} \; i_s=j_s,\; s=\{1:N\} 	\setminus \{k\}, \\
	\dfrac{i_k\sqrt{3}}{\sqrt{(2i_k+1)(2i_k-1)}}, &\mbox{if } i_k-1=j_k\; \hbox{and} \; i_s=j_s,\; s=\{1:N\} 	\setminus \{k\}, \\
		0, & \mbox{otherwise}.
	\end{cases}
\end{eqnarray*}
Hence, $\mathcal{G}_0$ is a identity matrix, whereas $\mathcal{G}_k, \; k>0$,  contains at most two nonzero entries per row; see, e.g., \cite{OGErnst_EUllmann_2010,CEPowell_HCElman_2009}. On the other hand, $\textbf{g}_i$ is the first column of $\mathcal{G}_i, \; i=0,1,\ldots,N$.


\section{Error estimates}\label{sec:error}

In this section, we  present  a priori error estimates for stationary convection diffusion equations with random coefficients, discretized by stochastic discontinuous Galerkin method.

Let a partition of the support of probability density in finite dimensional space, i.e.,  $\Gamma= \prod \limits_{n=1}^{N} \Gamma_n$ consists of a finite number of disjoint $\mathbf{R}^N$--boxes, $\gamma = \prod \limits_{n=1}^{N} (r_n^{\gamma},s_n^{\gamma})$, with $(r_n^{\gamma},s_n^{\gamma}) \subset \Gamma_n$ for $n=1,\ldots,N$. The mesh size $k_n$ is defined by $k_n = \max \limits_{\gamma} |s_n^{\gamma} - r_n^{\gamma}|$ for $1 \leq n \leq N$. For the multi--index $q = (q_1,\ldots,q_N)$, the (discontinuous) finite element approximation space with degree at most $q_n$ on each direction $\xi_n$ is denoted by $\mathcal{Y}_k^q \subset L^2(\Gamma)$. Then, for $v\in H^{q+1}(\Gamma), \varphi \in \mathcal{Y}_k^q$, we have the following estimate, see, e.g., \cite{IBabuska_RTempone_GEZouraris_2004a}
\begin{equation}\label{estr}
\min \limits_{\varphi \in \mathcal{Y}_k^q} \|v-\varphi\|_{L^2(\Gamma)} \leq   \sum_{n=1}^{N}\bigg(\frac{k_n}{2}\bigg)^{q_n+1}\dfrac{\lVert \partial^{q_n +1}_{\xi_n}v \rVert_{L^2(\Gamma)} }{(q_n+1)!}.
\end{equation}
To later use,  we  introduce the $L^2$--projection operator  $\Pi_n: L^2(\Gamma) \rightarrow \mathcal{Y}_k^q$ by
\begin{eqnarray}\label{eq:l2projc}
(\Pi_n(\xi)-\xi,\zeta)_{L^2(\Gamma)}=0 \qquad \quad \forall \zeta \in \mathcal{Y}_k^q, \qquad \forall\xi \in L^2(\Gamma),
\end{eqnarray}
and the $H^1$--projection operator $\mathcal{R}_h:H^1(\mathcal{D}) \rightarrow V_h \cap H^1(\mathcal{D})$ by
\begin{subequations}
\begin{eqnarray}
(\mathcal{R}_h( \nu)- \nu,\chi)_{L^2(\mathcal{D})}&=&0 \qquad \quad \forall \chi\in V_h, \qquad \forall \nu \in H^1(\mathcal{D}), \label{eq:Hprojc} \\
(\nabla(\mathcal{R}_h( \nu)- \nu),\nabla \chi)_{L^2(\mathcal{D})}&=&0  \qquad \quad  \forall \chi\in V_h, \qquad \forall \nu \in H^1(\mathcal{D}). \label{eq:Hproj}
\end{eqnarray}
\end{subequations}

Next, we  state the well--known trace and inverse inequalities, which are needed frequently  in the rest of the paper.
\begin{itemize}
\item For positive constant $c_{tr}$  independent of $ K \in \mathcal{T}_h$ and $h$, the trace  inequality is given as follow (see, e.g., \cite[Section~2.1]{BRiviere_2008a}):
\begin{subequations}\label{eq:trace}
	\begin{eqnarray}
		\lVert v \rVert_{E}^2  &\leq& c_{tr} \bigg(\lVert v \rVert_{L^2(K)}^2 + h_K\lvert  v \rvert_{H^1(K)}^2\bigg) \quad  v\in H^1(K),\\
  	   \lVert \nabla v\cdot \mathbf{n_E} \rVert_{E}^2  &\leq& c_{tr} \bigg(\lvert v \rvert_{H^1(K)}^2 + h_K\lvert  v \rvert_{H^2(K)}^2\bigg) \quad  \;\, v \in H^2(K).
	\end{eqnarray}
\end{subequations}
\item For positive constant $c_{inv}$ independent of $ K \in \mathcal{T}_h$ and $h$, the inverse inequality is given as follows (see, e.g., \cite[Section~4.5]{SCBrenner_LRScott_2008}):
\begin{eqnarray}\label{eq:inv}
|v|_{j,K} &\leq& c_{inv} \, h^{i-j} |v|_{i,K} \qquad \quad \forall v \in V_h, \qquad \; 0 \leq i \leq j \leq 2.
\end{eqnarray}
\end{itemize}

Lastly, we give discontinuous Galerkin approximation estimate for all $v\in H^2(K)$ for $ K\in \mathcal{T}_h$.
\begin{theorem} (\cite[Theorem 2.6]{BRiviere_2008a})
  Assume that $v\in H^2(K)$ for $K\in \mathcal{T}_h$ and  $\widetilde{v}  \in \mathbb{P}^\ell$. Then, there exists a constant $C$ independent of $v$ and $h$ such that
  \begin{equation} \label{approxfm}
    \lVert v- \widetilde{v} \rVert_{H^q(K)} \leq  C\,h^{\min(\ell+1,2)-q} \lvert v \rvert_{H^2(K)} \qquad 0\leq q \leq 2.
  \end{equation}
\end{theorem}

Let $\widetilde{u} \in V_h  \otimes \mathcal{Y}_k^q $ is an approximation of the solution $u$. We derive an approximation for the tensor product $V_h \otimes \mathcal{Y}_k^q$, which is a direct application of the results for $V_h$ and $\mathcal{Y}_k^q$ as done in \cite{IBabuska_RTempone_GEZouraris_2004a,KLiu_BRiviere_2013}.

\begin{theorem}\label{thm:bestapp}
Assume that $v \in L^2(H^2(\mathcal{D});\Gamma)\cap H^{q+1}(H^1(\mathcal{D});\Gamma)$  and  $\widetilde{v} \in V_h \otimes \mathcal{Y}_k^q $.  Then, we have
\begin{subequations}
\begin{eqnarray}
	 \lVert \nabla (v- \widetilde{v} )\rVert_{L^2(L^2(\mathcal{D});\Gamma)}\hspace{-3mm} &\leq&\hspace{-3mm}Ch^{\min(\ell+1,2)-1} \lVert v \rVert_{L^2(H^2(\mathcal{D});\Gamma)} \nonumber\\
	 &&\hspace{-3mm}+\sum_{n=1}^{N}\bigg(\frac{k_n}{2}\bigg)^{q_n+1}\dfrac{\lVert \partial^{q_n+1}_{\xi_n}v \rVert_{L^2(H^1(\mathcal{D});\Gamma)} }{(q_n+1)!}, \label{ineq:grad}	
\end{eqnarray}
\begin{eqnarray}
	\lVert \nabla^2 (v- \widetilde{v} )\rVert_{L^2(L^2(\mathcal{D});\Gamma)} \hspace{-3mm}&\leq&\hspace{-3mm}Ch^{\min(\ell+1,2)-2} \lVert v \rVert_{L^2(H^2(\mathcal{D});\Gamma)}\nonumber\\
	&& \hspace{-3mm}+ Ch^{-1}\sum_{n=1}^{N}\bigg(\frac{k_n}{2}\bigg)^{q_n+1}\dfrac{\lVert \partial^{q_n+1}_{\xi_n}v \rVert_{L^2(H^1(\mathcal{D});\Gamma)} }{(q_n+1)!}, \label{ineq:grad2}
\end{eqnarray}
\end{subequations}
where the constant $C$ independent of $v$, $h$, and $k_n$.
\end{theorem}
\begin{proof}
We refer to \ref{app:A} for the  proof of Theorem~\ref{thm:bestapp}.
\end{proof}

The next step is to use Theorem~\ref{thm:bestapp} together with the approximation estimate \eqref{approxfm} to derive an upper bound for the error in the energy norm.

\begin{theorem}\label{thm:main}
Assume $u  \in L^2(H^2(\mathcal{D});\Gamma) \cap H^{q+1}(H^1(\mathcal{D});\Gamma)$ and  $u_h \in V_h \otimes \mathcal{Y}_k^q$.
Then, there is a constant $C$ independent of $u, h$, and $k_n$ such that
\begin{eqnarray}
\lVert u-u_h\rVert_{\xi} &\leq&  C \left(  h^{\min(\ell+1,2)-1} \lVert u \rVert_{L^2(H^2(\mathcal{D});\Gamma)} \right.\nonumber\\
	 	&& \left.+\sum_{n=1}^{N}\bigg(\frac{k_n}{2}\bigg)^{q_n+1}\dfrac{\lVert \partial^{q_n+1}_{\xi_n}u \rVert_{L^2(H^1(\mathcal{D});\Gamma)} }{(q_n+1)!}\right).
\end{eqnarray}
\end{theorem}
\begin{proof}
We refer to \ref{app:B} for the  proof of Theorem~\ref{thm:main}.
\end{proof}

In practical implementations such as transport phenomena in random media, the length $N$ of the random vector $\xi$ can be  large, especially for the small correlation length in the covariance function of the random input. This increases the value of  multivariate stochastic basis polynomials $P$ quite fast, called as the curse of dimensionality. In the following section, we break the curse of dimensionality  by using low--rank  approximation, which reduces both the storage requirements and the computational complexity by exploiting a Kronecker--product structure of system matrices  defined in \eqref{tensormtrx}.


\section{Low--rank approximation}\label{sec:lowrank}

In this section, we develop efficient Krylov subspace solvers with suitable preconditioners where the solution is approximated using a low--rank representation in order to reduce memory requirements and computational effort. Basic operations associated with the low-rank format are much cheaper, and as Krylov subspace method converges, it constructs a sequence of low--rank approximations to the solution of the system.

We begin with the basic notation related to Kronecker products and low-rank approach. Let $\mathbf{u}=[u_1^T,\ldots,u_P^T]^T\in\mathbb{R}^{N_dP}$ with each $u_i$ of length $N_d$ and $\mathbf{U}=[u_1,\ldots,u_P]\in \mathbb{R}^{N_d \times P}$ where $N_d$ and $P$ are the degrees of freedom for the spatial discretization and the total degree of the multivariate stochastic basis polynomials, respectively. Then, we define isomorphic mappings between $\mathbb{R}^{N_dP}$ and $\mathbb{R}^{N_d \times P}$  as following
\[
\texttt{vec}:\mathbb{R}^{N_d \times P} \rightarrow \mathbb{R}^{N_dP},  \qquad \texttt{mat}: \mathbb{R}^{N_dP} \rightarrow \mathbb{R}^{N_d\times P}
\]
determined by  the operators \texttt{vec}($ \cdot $) and \texttt{mat}($ \cdot $), respectively. The matrix inner product is defined by $\ang{U,V}_F=\text{trace}(U^TV)$ so that the induced norm is $\|U\|_F = \sqrt{\ang{U,V}_F}$. For the sake of simplicity, we will omit the subscript in $\|\cdot\|_F$ and write only $\|\cdot\|$. Further, we have the following properties, see, e.g., \cite{DKressner_CTobler_2011}:
\begin{eqnarray*}
   \texttt{vec}(A\mathbf{U}B)=(B^T\otimes A)\texttt{vec}(\mathbf{U}),\qquad    (A\otimes B)(C\otimes D)= AC \otimes BD.
\end{eqnarray*}
Now, we can interpret  the system \eqref{tensormtrx}  as $ \mathbf{\mathcal{A}}(U) = \mathcal{F} $ for the matrix $ U \in
\mathbb{R}^{N_d\times P}$ with $\mathbf{u} = \text{vec}(\mathbf{U}) $, where $ \mathbf{\mathcal{A}}(\mathbf{U}) $ is defined as the linear operator satisfying $ \text{vec}(\mathbf{\mathcal{A}}(\mathbf{U})) = \mathbf{\mathcal{A}}\text{vec}(\mathbf{U}) $. Assuming low--rank decomposition of $\mathbf{U}= WV^T$ with
\begin{eqnarray*}
 W = [w_1,\ldots,w_k] \in \mathbb{R}^{N_d\times r}, \quad
 V = [v_1,\ldots,v_k] \in \mathbb{R}^{P\times r}, \quad r \ll N_d, P
\end{eqnarray*}
and
\begin{eqnarray*}
 \texttt{vec}(\mathbf{U}) = \texttt{vec} \Bigg( \sum_{i=1}^{r} w_i v_i^T \Bigg)  =  \sum_{i=1}^{r} v_i \otimes w_i,
\end{eqnarray*}
we have
\begin{eqnarray*}
  \mathbf{\mathcal{A}}\texttt{vec}(\mathbf{U}) & = & \Bigg( \sum_{k=0}^{N}\mathcal{G}_k\otimes \mathcal{K}_k \Bigg) \Bigg( \sum_{i=1}^{r} v_i \otimes w_i \Bigg) \\
  & = &   \sum_{k=0}^{N} \sum_{i=1}^{r} (\mathcal{G}_kv_i) \otimes  (\mathcal{K}_k w_i) \in \mathbb{R}^{N_d P}.
\end{eqnarray*}
This implies
\begin{eqnarray*}
   \mathbf{\mathcal{A}}(\mathbf{U}) :=  \texttt{mat}(\mathbf{\mathcal{A}}\texttt{vec}(\mathbf{U})) \in \mathbb{R}^{N_d\times P}.
\end{eqnarray*}

\begin{algorithm}[H]
    \scriptsize
	\caption{Low--rank preconditioned BiCGstab (LRPBiCGstab)}
	\label{alg:BiCGstab}
	\hspace*{\algorithmicindent}\textbf{Input:} Matrix functions $\mathcal{A}, \mathcal{P}: \mathbb{R}^{N_d\times P} \rightarrow \mathbb{R}^{N_d\times P}$, right--hand side
     $\mathcal{F}$ in  low--rank format. Truncation operator $\mathcal{T}$ w.r.t. given tolerance $\epsilon_{trunc}$.\\
	\hspace*{\algorithmicindent}\textbf{Output:} Matrix $\mathbf{U} \in \mathbb{R}^{N_d\times P}$ satisfying  $\|\mathcal{A}(\mathbf{U})-\mathcal{F}\| \leq \epsilon_{tol}$.\\
	\vspace{-5mm}
	\begin{algorithmic}[1]
		\STATE{$\mathbf{U}_0=0$, $R_0=\mathcal{F}$, $\widetilde{R}=\mathcal{F}$, $\rho_0=\langle \widetilde{R},R_0\rangle$, $S_0=R_0$, $\widetilde{S}_0=\mathcal{P}^{-1}(S_0)$, $V_0=\mathcal{A}(\widetilde{S}_0)$, $k=0$}
		\WHILE{$\lVert R_k\rVert>\epsilon_{tol}$}
		\STATE{$\omega_k=\langle \widetilde{R},R_k\rangle /\langle \widetilde{R},V_k\rangle$}
		\STATE{$Z_k=R_k-\omega_k V_k$, \hfill $Z_k\leftarrow \mathcal{T}(Z_k) $}
		\STATE{$\widetilde{Z}_k=\mathcal{P}^{-1}(Z_k)$, \hfill $\widetilde{Z}_k\leftarrow \mathcal{T}(\widetilde{Z}_k) $}
		\STATE{$T_k=\mathcal{A}(\widetilde{Z}_k)$, \hfill $T_k\leftarrow \mathcal{T}(T_k) $}
		\IF{$\lVert Z_k\rVert\leq \epsilon_{tol}$}
		\STATE {$\mathbf{U}=\mathbf{U}_{k} +\omega_k\widetilde{S}_k$}
		\RETURN
		\ENDIF
		\STATE{$\xi_{k}=\langle T_{k},Z_{k}\rangle/\langle T_{k},T_{k}\rangle$}
		\STATE{$\mathbf{U}_{k+1}=\mathbf{U}_{k}+\omega_k \widetilde{S}_k + \xi_{k}\widetilde{Z}_k$, \hfill $\mathbf{U}_{k+1}\leftarrow \mathcal{T}(\mathbf{U}_{k+1}^n) $}
		\STATE{$ R_{k+1} = \mathcal{F}- \mathcal{A}(\mathbf{U}_{k+1})$, \hfill $R_{k+1}\leftarrow \mathcal{T}(R_{k+1}) $ }
		\IF{$\lVert R_{k+1}\rVert \leq \epsilon_{tol}$}
		\STATE {$\mathbf{U}=\mathbf{U}_{k+1}$}
		\RETURN
		\ENDIF
		\STATE{$\rho_{k+1}=\langle \widetilde{R},R_{k+1}\rangle$}
		\STATE{$\beta_{k}=\frac{\rho_{k+1}}{\rho_{k}} \frac{\omega_k}{\xi_{k}}$}
		\STATE{$S_{k+1}=R_{k+1}+\beta_{k}(S_k-\xi_{k}V_k)$, \hfill $S_{k+1}\leftarrow \mathcal{T}(S_{k+1}) $}				
		\STATE{$\widetilde{S}_{k+1}=\mathcal{P}^{-1}(S_{k+1})$, \hfill $\widetilde{S}_{k+1}\leftarrow \mathcal{T}(\widetilde{S}_{k+1}) $}
		\STATE{$V_{k+1}=\mathcal{A}(\widetilde{S}_{k+1})$, \hfill $V_{k+1}\leftarrow \mathcal{T}(V_{k+1}) $}
		\STATE{$k=k+1$}
		\ENDWHILE
	\end{algorithmic}
\end{algorithm}
It is noted that in this study, we do not discuss the existence of the low--rank approximation. We refer to \cite{PBenner_AOnwunto_MStoll_2015,MAFreitag_DLHGreen_2018} and references therein.

We  here apply a variant of  Krylov subspace solvers, namely, conjugate gradient (CG) method \cite{MRHestenes_EStiefel_1952}, bi--conjugate gradient stabilized (BiCGstab) \cite{HAvanderVorst_1992}, quasi--minimal residual variant of the bi--conjugate gradient stabilized (QMRCGstab) method  \cite{TFChan_EGallopoulos_VSimoncini_TSzeto_CHTong_1992}, and generalized minimal residual (GMRES) \cite{YSaad_MHSchultz_1986} based on low--rank approximation, where the advantage is taken of the Kronecker product of the matrix $ \mathbf{\mathcal{A}}$. Algorithms~\ref{alg:BiCGstab}, ~\ref{alg:QMRCGstab}, and ~\ref{alg:GMRES} show a low--rank implementation of the classical left preconditioned BiCGstab, QMRCGstab, and GMRES methods, respectively.  We refer to \cite[Algorithm~1]{PBenner_AOnwunto_MStoll_2015} for low--rank variant of CG method. In principle, the low--rank truncation steps can affect the convergence of the Krylov method and the well--established properties of Krylov subspace  may no longer hold. Therefore, in the implementations, we  use a rather small truncation tolerance $\epsilon_{trunc}$ to try to maintain a very accurate representation of what the full--rank representation would like.

At each iteration step of the algorithm, we perform  truncation operators $\mathcal{T}$ and these operations substantially influence the overall solution procedure. The reason why we need to apply these operations is that  the rank of  low-rank factors can increase either via matrix vector products or vector (matrix) additions. Thus, rank--reduction techniques are required to keep costs under control, such as truncation based on singular values \cite{DKressner_CTobler_2011} or truncation based on coarse--grid rank reduction \cite{KLee_HCElman_2017}. In this paper, following the discussion in  \cite{MStoll_TBreiten_2015,PBenner_AOnwunto_MStoll_2015}, a more economical alternative  could be possible to compute singular values  a truncated SVD of  $ U= W^TV \approx B \diag(\sigma_{1},\dots,\sigma_{r})C^T$  associated to the $r$ singular values that are larger than the given truncation threshold. In this way, we obtain the new low--rank representation $U \approx \widetilde{U} \widetilde{V}^T$ by keeping both the rank of  low-rank factor and cost under control.

The  inner product computations in the iterative algorithms can  be done easily by applying the following strategy:
\[
\ang{Y,Z}=\texttt{vec}(Y)^T\texttt{vec}(Z)=\text{trace}(Y^TZ)
\]
for the low-rank matrices
\begin{eqnarray*}
	Y &=& W_Y V_Y^T \quad W_Y \in \mathbb{R}^{N_d\times r_y}, \; V_Y \in \mathbb{R}^{P\times r_y},  \\
	Z &=& W_Z V_Z^T \quad W_Z \in \mathbb{R}^{N_d\times r_Z}, \; V_Z \in \mathbb{R}^{P\times r_Z}.
\end{eqnarray*}
Then, one can easily show that
\[
\text{trace}(Y^TZ) = \text{trace} \Bigg( (W_YV_Y^T)^T (W_ZV_Z^T)\Bigg) = \text{trace}\Bigg( (V_Z^TV_Y) (W_Y^TW_Z)\Bigg)
\]
allows us to compute the trace of small matrices rather than of the ones from the full discretization.

\begin{algorithm}[H]
    \scriptsize
	\caption{Low--rank preconditioned QMRCGstab (LRPQMRCGstab)}
	\label{alg:QMRCGstab}
	\hspace*{\algorithmicindent}\textbf{Input:} Matrix functions $\mathcal{A}, \mathcal{P}: \mathbb{R}^{N_d\times P} \rightarrow \mathbb{R}^{N_d\times P}$, right--hand side $\mathcal{F}$ in  low--rank format. Truncation operator $\mathcal{T}$ w.r.t. given tolerance $\epsilon_{trunc}$.\\
	\hspace*{\algorithmicindent}\textbf{Output:} Matrix $\mathbf{U} \in \mathbb{R}^{N_d\times P}$ satisfying  $\lVert\mathcal{A}(\mathbf{U})-\mathcal{F}\rVert \leq \epsilon_{tol}$.\\
	\vspace{-5mm}
	\begin{algorithmic}[1]
		\STATE{$R_0=\mathcal{F} -\mathcal{A}(\mathbf{U}_0)$, for some initial guess $ \mathbf{U}_0 $.}
		\STATE{$Z_0=\mathcal{P}^{-1}(R_0)$}
		\STATE{Choose $\widetilde{R}_0$ such that $\ang{Z_0,\widetilde{R}_0}\neq 0$ (for example, $\widetilde{R}_0=R_0$).}
		\STATE{$ Y_0=V_0=D_0=0 $}
		\STATE{$ \rho_0=\alpha_0=\omega_0=1, \tau_0=\lVert Z_0\rVert_F, \theta_0=0, \eta_0=0, k =0$}				
		\WHILE{$ \sqrt{k+1}\lvert\widetilde{\tau}\rvert/\lVert R_0\rVert>\epsilon_{tol}$}
		\STATE{$\rho_{k+1}=\ang{Z_k, \widetilde{R}_0}$,\; $\beta_{k+1}=\frac{\rho_{k+1}}{\rho_{k}} \frac{\alpha_{k}}{\omega_{k}} $}	
		\STATE{$Y_{k+1}=Z_k + \beta_{k+1}(Y_k-\omega_{k}V_k)$, \hfill $Y_{k+1}\leftarrow \mathcal{T}(Y_{k+1}) $}
		\STATE{$\widetilde{Y}_{k+1}=\mathcal{A}(Y_{k+1})$, \hfill $\widetilde{Y}_{k+1}\leftarrow \mathcal{T}(\widetilde{Y}_{k+1}) $}	
		\IF{$\lVert \widetilde{Y}_{k+1}\rVert\leq \epsilon_{tol}$}
		\STATE {$\mathbf{U}=\mathbf{U}_{k}$}
		\RETURN
		\ENDIF
		\STATE{$V_{k+1}=\mathcal{P}^{-1}(\widetilde{Y}_{k+1})$, \hfill $V_{k+1}\leftarrow \mathcal{T}(V_{k+1}) $}
		\STATE{$\alpha_{k+1}=\rho_{k+1}/\ang{V_{k+1},\widetilde{R}_0}$}
		\STATE{$S_{k+1}=Z_k-\alpha_{k+1}V_{k+1}$, \hfill $S_{k+1}\leftarrow \mathcal{T}(S_{k+1}) $}
        \STATE{$\widetilde{\tau}=\tau \widetilde{\theta}_{k+1}c, \widetilde{\eta}_{k+1}=c^2\alpha_{k+1}$}	
        \STATE{$\widetilde{D}_{k+1}=Y_{k+1}+\dfrac{\theta_{k}^2\eta_{k}}{\alpha_{k+1}}D_k$, \hfill $\widetilde{D}_{k+1}\leftarrow \mathcal{T}(\widetilde{D}_{k+1}) $}
        \STATE{$\widetilde{\mathbf{U}}_{k+1}=\mathbf{U}_k+\widetilde{\eta}_{k+1}\widetilde{D}_{k+1}$, \hfill $\widetilde{\mathbf{U}}_{k+1}\leftarrow \mathcal{T}(\widetilde{\mathbf{U}}_{k+1}) $}	
		\STATE{$\widetilde{S}_{k+1}=\mathcal{A}(S_{k+1})$, \hfill $\widetilde{S}_{k+1}\leftarrow \mathcal{T}(\widetilde{S}_{k+1})$}
		\STATE{$T_{k+1}=\mathcal{P}^{-1}(\widetilde{S}_{k+1})$, \hfill $T_{k+1}\leftarrow \mathcal{T}(T_{k+1}) $}
		\STATE{$\omega_{k+1}= \ang{S_{k+1},T_{k+1}}/\ang{T_{k+1},T_{k+1}} $} 		
		\STATE{$Z_{k+1}=S_{k+1}-\omega_{k+1}T_{k+1} $} 			
	    \STATE{$\theta_{k+1}=\lVert Z_{k+1}\rVert/\widetilde{\tau},\; c=\dfrac{1}{\sqrt{1+\theta_{k+1}^2}}$}	
	    \STATE{$\tau=\widetilde{\tau} \theta_{k+1}c, \eta_{k+1}=c^2\omega_{k+1}$}	
		\STATE{$D_{k+1}=S_{k+1}+\dfrac{\widetilde{\theta}_{k+1}^2\widetilde{\eta}_{k+1}}{\omega_{k+1}}\widetilde{D}_{k+1}$, \hfill $D_{k+1}\leftarrow \mathcal{T}(D_{k+1}) $}
		\STATE{$\mathbf{U}_{k+1}=\widetilde{\mathbf{U}}_{k+1}+\eta_{k+1}D_{k+1}$, \hfill $\mathbf{U}_{k+1}\leftarrow \mathcal{T}(\mathbf{U}_{k+1}) $}	
		\STATE{$k=k+1$}
		\ENDWHILE
		\STATE{$\mathbf{U}=\mathbf{U}_{k}$}
	\end{algorithmic}
\end{algorithm}
It is well--known that Krylov subspace methods require preconditioning in order to obtain a fast convergence in terms of the number of iterations and low--rank Krylov methods have no exception. However, the precondition operator must not dramatically increase the memory requirements of the solution process, while it reduces the number of iterations at a reasonable computational cost. We present here the well--known preconditioners:
\begin{algorithm}[H]
    \scriptsize
	\caption{Low--rank  preconditioned GMRES (LRPGMRES)}
	\label{alg:GMRES}
	\hspace*{\algorithmicindent}{\textbf{Input:} Matrix functions $\mathcal{A}, \mathcal{P}: \mathbb{R}^{N_d\times P} \rightarrow \mathbb{R}^{N_d\times P}$, right--hand side $\mathcal{F}$ in low--rank format. Truncation operator $\mathcal{T}$ w.r.t. given tolerance $\epsilon_{trunc}$.}\\
	\hspace*{\algorithmicindent}{\textbf{Output:} Matrix $\mathbf{U} \in \mathbb{R}^{N_d\times P}$ satisfying $\lVert \mathcal{A}(\mathbf{U})-\mathcal{F} \rVert \leq \epsilon_{tol}$.}\\
	\vspace{-5mm}
	\begin{algorithmic}[1]
		\STATE{$R_0=\mathcal{F} -\mathcal{A}(\mathbf{U}_0)$, for some initial guess $ \mathbf{U}_0$.}
		\STATE{$V_1=R_0/\lVert R_0\rVert$}
		\STATE{$\xi=[\xi_1,0,\ldots,0]$, \qquad  $\xi_1=\lVert V_1\rVert$}		
        \FOR{$ k=1,\ldots,\text{maxit}$}
        \STATE{$Z_{k}=\mathcal{P}^{-1}(V_{k})$, \hfill $Z_{k}\leftarrow \mathcal{T}(Z_{k}) $}
        \STATE{$W=\mathcal{A}(Z_{k})$, \hfill $W\leftarrow \mathcal{T}(W) $}
        \FOR{$ i=1,\ldots,k$}
        \STATE{$h_{i,k}=\ang{W,V_i}$}
        \STATE{$W=W-h_{i,k}V_i$, \hfill $W\leftarrow \mathcal{T}(W) $}	
        \ENDFOR
        \STATE{$h_{k+1,k}=\lVert W\rVert$}
        \STATE{$V_{k+1}=W/h_{k+1,k}$}
        \STATE{Apply Givens rotations to kth column of $h$, i.e.,}
        \FOR{$ i=1,\ldots,k-1$}
        \STATE{$\left[ \begin{matrix}
        	h_{i,k}\\
        	h_{i+1,k}
        	\end{matrix}\right] =
        	\left[ \begin{matrix}
        	c_i & s_i\\
        	-s_i & c_i
        	\end{matrix}\right]\left[ \begin{matrix}
        	h_{i,k}\\
        	h_{i+1,k}
        	\end{matrix}\right]  $}
        \ENDFOR
        \STATE{Compute kth rotation, and apply to $\xi$ and last column of $h$.\\
        $\left[ \begin{matrix}
        h_{i,k}\\
        h_{i+1,k}
        \end{matrix}\right] =
        \left[ \begin{matrix}
        c_i & s_i\\
        -s_i & c_i
        \end{matrix}\right]\left[ \begin{matrix}
        h_{i,k}\\
        h_{i+1,k}
        \end{matrix}\right] $}
        \STATE{$ h_{k,k}=c_kh_{k,k}+s_kh_{k+1,k} $, \qquad  $h_{k+1,k}=0$}
        \IF{$ \lvert\xi_{k+1}\rvert $ sufficiently small}
        \STATE{Solve $Hy=\xi$, where the entries of $H$ are $h_{j,k}$.}
        \STATE{$Y=[y_1V_1,\ldots,y_kV_k]$, \hfill $Y\leftarrow \mathcal{T}(Y) $}
        \STATE{$\widetilde{Y}=\mathcal{P}^{-1}(Y)$, \hfill $\widetilde{Y}\leftarrow \mathcal{T}(\widetilde{Y}) $}        	
        \STATE{$\mathbf{U}=\mathbf{U}_{0}+\widetilde{Y}$, \hfill $\mathbf{U}^n\leftarrow \mathcal{T}(\mathbf{U})$}
        \RETURN
        \ENDIF      	
        \ENDFOR
	\end{algorithmic}
\end{algorithm}
\begin{itemize}
\item[i)] Mean-based preconditioner
\begin{equation*}
\mathcal{P}_0=\mathcal{G}_0 \otimes \mathcal{K}_0
\end{equation*}
is one of the most commonly used preconditioners for solving PDEs with random data, see, e.g., \cite{CEPowell_HCElman_2009,RGGhanem_RMKruger_1996}. One can easily observe that $\mathcal{P}_0$ is block diagonal matrix since $\mathcal{G}_0$ is a diagonal matrix due to the orthogonality of the stochastic basis functions $\psi_i$.

\item[ii)] Ullmann preconditioner, which is of the form
\begin{equation*}
\mathcal{P}_1=\underbrace{\mathcal{G}_0\otimes\mathcal{K}_0}_{:=\mathcal{P}_0} + \sum_{k=1}^{N} \dfrac{\text{trace}(\mathcal{K}_k^T \mathcal{K}_0)}{\text{trace}(\mathcal{K}_0^T\mathcal{K}_0)} \mathcal{G}_k\otimes \mathcal{K}_0,
\end{equation*}
can be considered as a modified version of $\mathcal{P}_0$, see, e.g., \cite{EUllmann_2010}. One of the advantages of this preconditioner is keeping the structure of the coefficient matrix, which in this case, sparsity pattern. Moreover, unlike the mean--based preconditioner, it uses the whole information in the coefficient matrix. However, this advantage causes $ \mathcal{P}_1$ being more expensive since it is not block diagonal anymore.
\end{itemize}

\section{Unsteady model problem with random coefficients}\label{sec:unsteady}

In this section, we extend our discussion to  unsteady convection diffusion equation with random coefficients: find  $u: \overline{\mathcal{D}} \times \Omega \times [0,T] \rightarrow \mathbb{R}$ such that $\mathbb{P}$-almost surely in $\Omega$
\begin{subequations}\label{eqn:unsteady}
	\begin{eqnarray}
	\dfrac{\partial u(x,\omega,t)}{\partial t}-\nabla\cdot(a(x,\omega)\nabla u(x,\omega,t))\nonumber\\
	+ \mathbf{b}(x,\omega)\cdot \nabla u(x,\omega,t)  & = & f(x,t), \; \; \hbox{  in} \;\; \mathcal{D} \times \Omega \times (0,T], \\
	u(x,\omega,t) & = & 0, \qquad \quad  \;  \hbox{on} \;\; \partial \mathcal{D} \times \Omega \times [0,T],\\
	u(x,\omega,0) & = & u^0(x), \quad \;\;  \hbox{in} \;\; \mathcal{D} \times \Omega,
	\end{eqnarray}
\end{subequations}
where $u^0(x) \in L^2(\mathcal{D})$ corresponds to deterministic initial condition.

By following the methodologies introduced for the stationary problem in Section~\ref{sec:model} and backward Euler method in temporal space with the uniform time step  $\Delta t=T/N$, we obtain the following system of ordinary equations with block structure:
\begin{eqnarray*}
\big(\mathcal{G}_0\otimes M\big) \bigg(\dfrac{\mathbf{u}^{n+1}-\mathbf{u}^{n}}{\Delta t}\bigg)  + \bigg(\sum_{k=0}^{N}\mathcal{G}_k\otimes \mathcal{K}_k\bigg) u^{n+1}= \bigg(g_0 \otimes f_0\bigg)^{n+1},
\end{eqnarray*}
or, equivalently,
\begin{eqnarray}\label{eq:fully2}
\mathcal{M} \bigg(\dfrac{\mathbf{u}^{n+1}-\mathbf{u}^{n}}{\Delta t}\bigg) +  A\mathbf{u}^{n+1}= F^{n+1}, 
\end{eqnarray}
where \begin{eqnarray*}
A =  \sum_{k=0}^{N}\mathcal{G}_k\otimes \mathcal{K}_k, \quad \mathcal{M}=\mathcal{G}_0\otimes M, \quad
F^{n+1} =\bigg(g_0 \otimes f_0\bigg)^{n+1}.
\end{eqnarray*}
Rearranging the \eqref{eq:fully2}, we obtain the following matrix form of the discrete systems:
\begin{equation}\label{eq:linear}
\mathbf{\mathcal{A}} \mathbf{u}^{n+1}=\mathcal{F}^{n+1},
\end{equation}
where for  $k=1,\ldots,N$
\begin{eqnarray*}
	\mathbf{\mathcal{A}} &=& \mathcal{G}_0\otimes \underbrace{(M +\Delta t\mathcal{K}_0)}_{\widehat{\mathcal{K}}_0} + \bigg(\sum_{k=1}^{N}\mathcal{G}_k \otimes \underbrace{(\Delta t \mathcal{K}_k)}_{\widehat{\mathcal{K}}_k} \bigg),\\
	\mathcal{F}^{n+1} &=&  \mathcal{M} \mathbf{u}^{n} + \Delta t F^{n+1}.
\end{eqnarray*}

Next, we state  the stability analysis of the proposed method on the energy norm defined in \eqref{energynorm}.

\begin{theorem}\label{thm:stability}
	There exists a constant C independent of $h$ and $\Delta t$ such that for all $m>0$
	\begin{eqnarray*}
		\|u^{m}\|^2_{L^2(L^2(\mathcal{D});\Gamma)}  + \Delta t \sum_{n=1}^{m}\|u^{n}\|^2_{\xi}	\leq  C \bigg( \|u^{0}\|^2_{L^2(L^2(\mathcal{D});\Gamma)} + \Delta t\sum_{n=1}^{m} \|f^{n}\|^2_{L^2(L^2(\mathcal{D});\Gamma)}\bigg).
	\end{eqnarray*}
\end{theorem}
\begin{proof}
Taking $v=u^{n+1}$ in the following fully discrete system
\begin{eqnarray}\label{eqn:fully}
    \frac{1}{\Delta t}   \int \limits_{\Gamma} \int \limits_{\mathcal{D}}  ( u^{n+1}-u^n) \, v \; dx \; \rho(\xi) \, d\xi   + a_{\xi}(u^{n+1},v)  = l_{\xi}(t_{n+1},v)
\end{eqnarray}
we obtain
\begin{eqnarray*}
	\frac{1}{\Delta t}   \int \limits_{\Gamma} \int \limits_{\mathcal{D}}  ( u^{n+1}-u^n) \, u^{n+1} \; dx \; \rho(\xi) \, d\xi  	 + a_{\xi}(u^{n+1},u^{n+1})= l_{\xi}(t_{n+1},u^{n+1}).
\end{eqnarray*}
An application of the polarization identity
\begin{eqnarray*}\label{ineq:polar}
	\forall x,y\in \mathbb{R}, \quad \frac{1}{2}(x^2-y^2) \leq \frac{1}{2}(x^2-y^2+(x-y)^2) = (x-y)x,
\end{eqnarray*}
yields
\begin{eqnarray}\label{eq:111}
	\frac{1}{2\Delta t} \bigg( \|u^{n+1}\|^2_{L^2(L^2(\mathcal{D});\Gamma)} - \|u^{n}\|^2_{L^2(L^2(\mathcal{D});\Gamma)} \bigg) + a_{\xi}(u^{n+1},u^{n+1})  = l_{\xi}(t_{n+1},u^{n+1}).
\end{eqnarray}
From the coercivity of $a_{\xi}$ \eqref{coer}, Cauchy-Schwarz's, and Young's inequalities,  the expression \eqref{eq:111} reduces to
\begin{align*}
& \frac{1}{2\Delta t} \bigg( \|u^{n+1}\|^2_{L^2(L^2(\mathcal{D});\Gamma)} - \|u^{n}\|^2_{L^2(L^2(\mathcal{D});\Gamma)}\bigg) + \frac{c_{cv}}{2}\|u^{n+1}\|^2_{\xi}   \leq |l_{\xi}(t_{n+1},u^{n+1})| \\
& \quad  \leq \|f^{n+1}\|_{L^2(L^2(\mathcal{D});\Gamma)} \|u^{n+1}\|_{L^2(L^2(\mathcal{D});\Gamma)}\nonumber \\
& \quad  \leq \frac{1}{2}\|f^{n+1}\|^2_{L^2(L^2(\mathcal{D});\Gamma)} + \frac{1}{2}\|u^{n+1}\|^2_{L^2(L^2(\mathcal{D});\Gamma)}.
\end{align*}
Multiplying by $2\Delta t$ and summing from $n=0$ to $n=m-1$, we obtain
\begin{eqnarray*}
		\|u^{m}\|^2_{L^2(L^2(\mathcal{D});\Gamma)} &-& \|u^{0}\|^2_{L^2(L^2(\mathcal{D});\Gamma)} + \Delta t c_{cv}\sum_{n=1}^{m}\|u^{n}\|^2_{\xi} \\
		&\leq&  \Delta t\sum_{n=1}^{m} \|f^{n}\|^2_{L^2(L^2(\mathcal{D});\Gamma)} + \Delta t \sum_{n=1}^{m}\|u^{n}\|^2_{L^2(L^2(\mathcal{D});\Gamma)}.
\end{eqnarray*}
After applying discrete Gronwall inequality \cite{BRiviere_2008a}, the desired result is obtained
\begin{eqnarray*}
		\|u^{m}\|^2_{L^2(L^2(\mathcal{D});\Gamma)}  + \Delta t \sum_{n=1}^{m}\|u^{n}\|^2_{\xi}	\leq  C \bigg( \|u^{0}\|^2_{L^2(L^2(\mathcal{D});\Gamma)} + \Delta t\sum_{n=1}^{m} \|f^{n}\|^2_{L^2(L^2(\mathcal{D});\Gamma)}\bigg),
\end{eqnarray*}
where the constant $C$ is independent of $h$ and $\Delta t$.
\end{proof}

Ones can easily derive a priori error estimates for unsteady stochastic problem \eqref{eqn:unsteady} by the following procedure as done for the stationary problem in Section~\ref{sec:error}. We also note that time dependence of the problem introduces additional complexity of solving a large linear system for each time step. Therefore, we apply the low--rank approximation technique introduced in Section~\ref{sec:lowrank} for each fixed time step.


\section{Numerical Results}\label{sec:num}

In this section, we present several numerical results to examine the quality of the proposed numerical approaches.  As mentioned before, we are here  interested in the quality of interest moments of the solution $u(x,\omega)$ in \eqref{eqn:m1} rather than the solution  $u(x,\omega)$. The numerical experiments are performed on an Ubuntu Linux machine with 32 GB RAM using MATLAB R2020a. To compare the performance of the solution methods, we report  the rank of the computed solution, the number of performed iterations, the computational time, the relative residual, that is,  $\|\mathcal{A} \mathbf{u} -  \mathcal{F}\| / \|\mathcal{F}\|$, and the memory demand of the solution.  Unless otherwise stated, in all simulations, iterative methods are terminated when the residual, measured in the Frobenius norm, is reduced to $\epsilon_{tol} = 10^{-4}$ or  the maximum iteration number ($\#iter_{max} =100$)  is reached. We note that the tolerance $\epsilon_{tol}$ should be chosen, such that  $\epsilon_{trunc} \leq  \epsilon_{tol}$;  otherwise, one would be essentially iterating on the noise from the low--rank truncations.

In the numerical experiments, the random input $\eta$ is characterized by the covariance function
\begin{align}\label{Cov:Gauss}
 C_{\eta} (\mathbf{x},\mathbf{y}) = \kappa^2 \prod_{n=1}^{2}  e^{-\left| x_n -y_n \right|/\ell_n } \quad \forall (\mathbf{x},\mathbf{y}) \in \mathcal{D}
\end{align}
with the correlation length $\ell_n$.  We use linear elements to generate discontinuous Galerkin basis and Legendre polynomials as stochastic basis functions since the underlying random variables have uniform distribution  over $[-\sqrt{3},\sqrt{3}]$. The eigenpair $(\lambda_j, \phi_j)$ corresponding to covariance function \eqref{Cov:Gauss} are given explicitly in \cite{GJLord_CEPowell_TShardlow_2014}.


\subsection{Stationary problem with random diffusion parameter}\label{ex:stationary_diff}
As a first benchmark problem, we consider a two-dimensional stationary convection diffusion equation with random diffusion parameter \cite{KLee_HCElman_2017} defined on $\mathcal{D} = [-1,1]^2$ with the deterministic source function $f(x)=0$, the constant convection parameter $ \mathbf{b}(x)=(0,1)^T$, and the Dirichlet boundary condition
\[
u_d(x)=\begin{cases}
u_d(x_1,-1)=x_1, & u_d(x_1,1)=0,\\
u_d(-1,x_2)=-1, & u_d(1,x_2)=1.
\end{cases}
\]
The random diffusion parameter is defined by $a(x,\omega)= \nu \,\eta(x,\omega)$, where the random field $\eta(x,\omega)$ can be chosen as a uniform random field having unity mean with the corresponding covariance function \eqref{Cov:Gauss} and $\nu$ is the viscosity parameter.  The solution exhibits exponential boundary layer near $x_2 =1$, where the value of the solution changes dramatically. Therefore, discontinuous Galerkin discretization in the spatial domain can be a better alternative compared to standard finite element methods; see Figure~\ref{fig:SDmeanVariance} for the mean and variance of solutions  for various values of viscosity parameter $\nu$. As $\nu$ decreases, the boundary layer becomes more visible.

\indent Table~\ref{tab::RD_SDG_N}, ~\ref{tab::RD_SDG_Vis}, and ~\ref{tab::RD_SDG_KappaPrecond} report the results of the simulations  by considering  various data sets. We show results for varying truncation number in KL expansion $N$, while keeping other parameters constant in Table~\ref{tab::RD_SDG_N}.
\begin{figure}[t]
	\centering
	\includegraphics[width=1\textwidth]{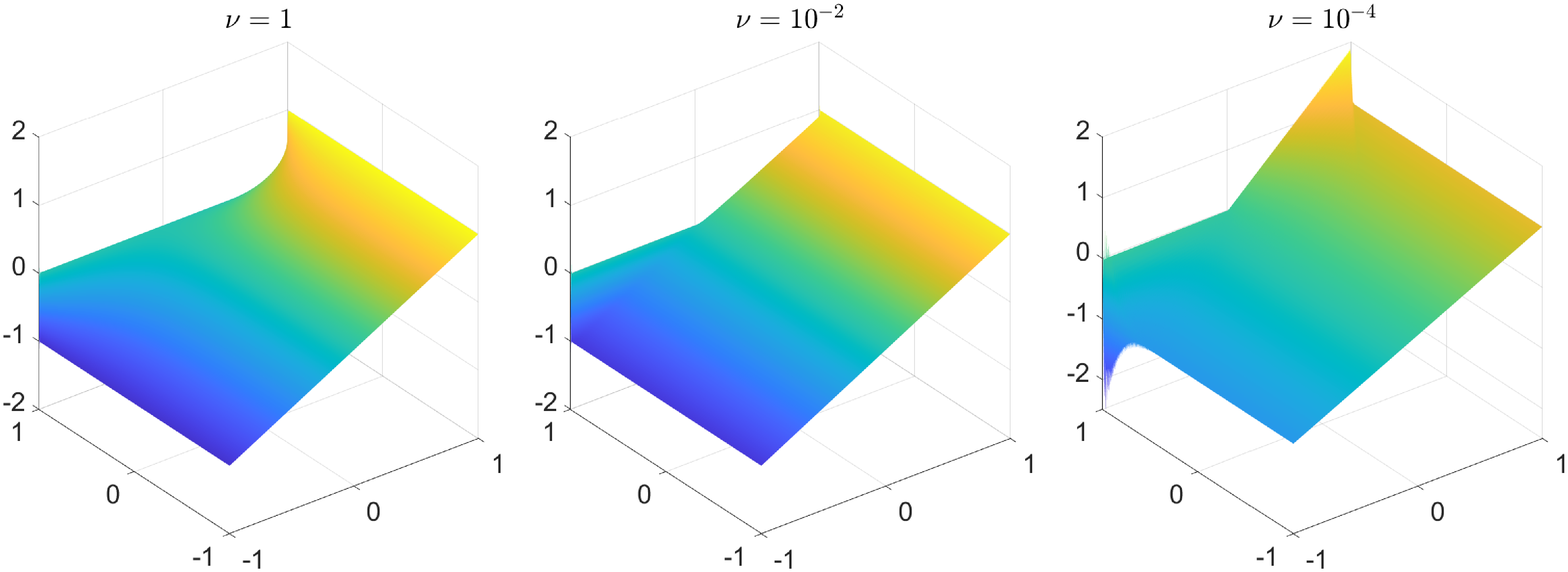}
	\includegraphics[width=1\textwidth]{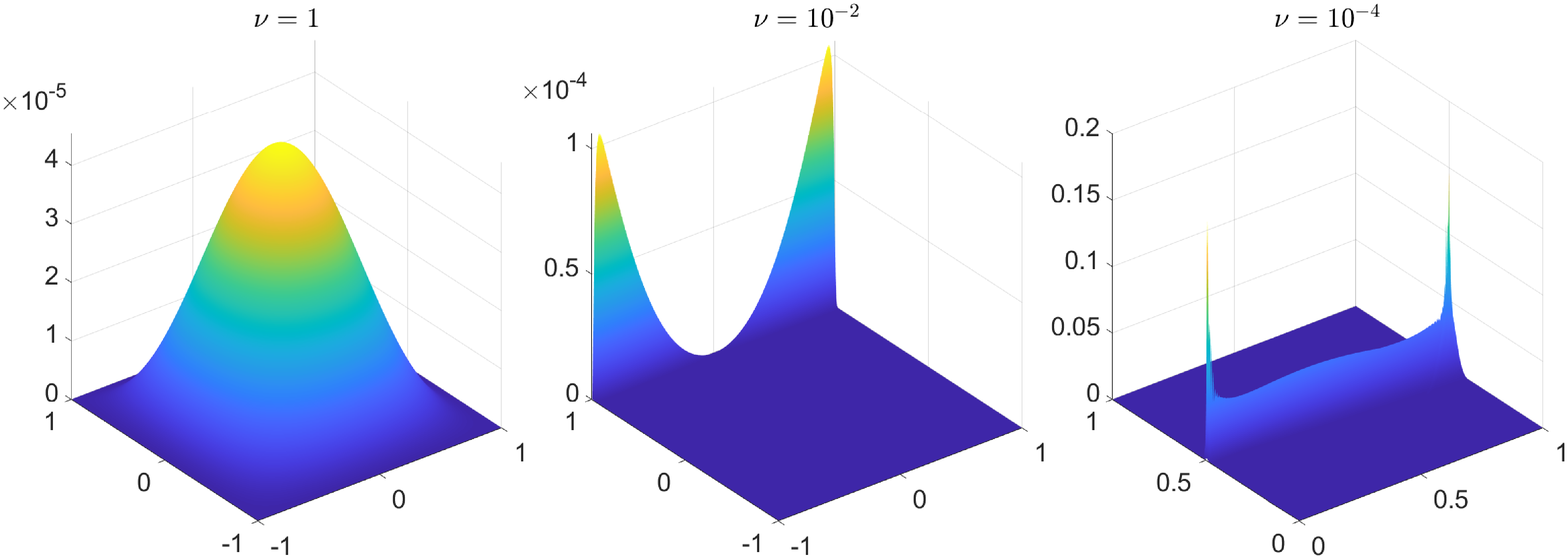}	
	\caption{Example~\ref{ex:stationary_diff}: Mean (top) and variance (bottom) of SG solutions obtained solving by $\mathcal{A}\backslash \mathcal{F}$ with $\ell=1$, $\kappa=0.05$, $N_d=393216$, $N=3$, and $Q=2$ for various values of viscosity parameter $\nu$.}
	\label{fig:SDmeanVariance}		
\end{figure}

\begin{table}[H]
	\scriptsize{
	\caption{Example~\ref{ex:stationary_diff}: Simulation results showing ranks of truncated solutions, total number of iterations, total CPU times (in seconds),  relative residual, and memory demand of the solution (in KB) with $N_d=6144 $, $Q=3$, $\ell=1$, $\kappa=0.05$,  $\nu=10^{-4}$, and  the mean-based preconditioner $\mathcal{P}_0$ for varying values of $N$.}
	\label{tab::RD_SDG_N}
	\hspace{-15mm}
	\begin{tabular}{ccccc}
	\begin{tabular}[c]{@{}c@{}}Method\\ $\epsilon_{trunc}$\end{tabular}	& \begin{tabular}[c]{@{}c@{}}LRPCG\\ 1e-06 (1e-08)\end{tabular} & \begin{tabular}[c]{@{}c@{}}LRPBiCGstab\\ 1e-06 (1e-08)\end{tabular} &  \begin{tabular}[c]{@{}c@{}}LRPQMRCGstab\\ 1e-06 (1e-08)\end{tabular} & \begin{tabular}[c]{@{}c@{}}LRPGMRES\\ 1e-06 (1e-08)\end{tabular} \\ \hline
		\hline
		N=3    &                         &                         &                         &    \\
		Ranks  & 10 (10)                 & 10 (10)                 & 9 (10)                  & 10 (10) \\
		\#iter & 5 (5)                   & 3 (3)                   & 3 (3)                   & 4 (4)   \\
		CPU    & 7.0 (7.7)               & 8.8 (8.9)               & 7.8 (10.0)              & 5.4 (5.2)  \\
		Resi.  & 1.8160e-07 (3.2499e-07) & 5.5982e-06 (5.5413e-06) & 2.9222e-05 (3.1021e-05) & 6.3509e-07 (6.3509e-07) \\
		Memory & 481.6 (481.6)           & 481.6 (481.6)           & 433.4 (481.6)           & 481.6 (481.6) \\
		\hline
		\hline
		N=4    &                         &                         &                         &    \\
		Ranks  & 12 (18)                 & 17 (18)                 & 17 (17)                 & 17 (17) \\
		\#iter & 4 (5)                   & 3 (3)                   & 3 (3)                   & 5 (5)   \\
		CPU    & 9.2 (13.3)              & 15.3 (15.3)             & 14.2 (14.4)             & 12.1 (11.7)  \\
		Resi.  & 1.2367e-06 (1.4311e-07) & 7.7090e-06 (7.7030e-06) & 1.0819e-05 (4.3167e-06) & 8.1316e-08 (8.1316e-08) \\
		Memory & 579.3 (868.9)           & 820.7 (868.9)           & 820.7 (820.7)           & 820.7 (820.7) \\
		\hline
		\hline
		N=5    &                         &                         &                         &    \\
		Ranks  & 18 (28)                 & 21 (28)                 & 22 (28)                 & 19 (28) \\
		\#iter & 4 (5)                   & 3 (3)                   & 3 (3)                   & 5 (5)  \\
		CPU    & 15.2 (20.8)             & 24.9 (25.1)             & 25.4 (26.0)             & 20.3 (20.6) \\
		Resi.  & 1.1705e-06 (8.2045e-08) & 8.5525e-06 (8.5527e-06) & 1.6985e-06 (8.6182e-07) & 8.4680e-08 (8.4680e-08) \\
		Memory & 871.9 (1356.3)          & 1017.2 (1356.3)         & 1065.6 (1356.3)         & 920.3 (1356.3) \\
		\hline
		\hline
		N=6    &                         &                         &                         &    \\
		Ranks  & 26 (42)                 & 26 (42)                 & 25 (42)                 & 25 (42) \\
		\#iter & 4 (4)                   & 3 (3)                   & 3 (3)                   & 4 (4)   \\
		CPU    & 25.6 (32.0)             & 42.4 (43.7)             & 49.4 (51.2)             & 29.7 (31.5)  \\
		Resi.  & 9.2495e-07 (1.0605e-06) & 9.6694e-06 (9.6649e-06) & 7.7812e-07 (4.1770e-07) & 1.0476e-06 (1.0476e-06) \\
		Memory & 1265.1 (2043.6)         & 1265.1 (2043.6)         & 1216.4 (2043.6)         & 1216.4 (2043.6) \\
		\hline
		\hline
		N=7    &                         &                         &                         &    \\
		Ranks  & 30 (60)                 & 32 (60)                 & 32 (60)                 & 28 (47)  \\
		\#iter & 4 (4)                   & 3 (3)                   & 3 (3)                   & 4 (4)    \\
		CPU    & 52.5 (58.8)             & 69.3 (73.8)             & 86.1 (87.9)             & 57.9 (57.7)  \\
		Resi.  & 1.0719e-06 (1.1205e-06) & 9.9865e-06 (9.9880e-06) & 6.5595e-07 (2.0226e-07) & 1.1075e-06 (1.1075e-06)  \\
		Memory & 1468.1 (2936.3)         & 1566 (2936.3)           & 1566 (2936.3)           & 1370.3 (2300.1)  \\
		\hline
		\hline	
	\end{tabular}
}
\end{table}

When $N$ increases, the complexity of the problem  increases. As expected,  decreasing the truncation tolerance $\epsilon_{trunc}$  increases the cost of computational time and memory requirement, especially for large $N$.  Another key observation from the Table~\ref{tab::RD_SDG_N} is that LRPGMRES exhibits better performance compared to other iterative solvers in terms of CPU time and memory requirement. Table~~\ref{tab::RD_SDG_Vis} displays the performance of low--rank of  Krylov subspace methods with the mean--based  preconditioner $\mathcal{P}_0$ for varying viscosity parameter $\nu$. Decreasing the values of $\nu$ makes the problem more convection dominated. Thus, the rank of the low--rank solution and memory requirements  increase for all iterative solvers.
\begin{table}[t]
	\scriptsize{
	\caption{Example~\ref{ex:stationary_diff}: Simulation results showing ranks of truncated solutions, total number of iterations, total CPU times (in seconds),  relative residual, and memory demand of the solution (in KB) with $N_d=6144 $, $Q=3$, $\ell=1$, $\kappa=0.05$,  $N=7$, and the mean-based preconditioner $\mathcal{P}_0$ for various values of viscosity parameter $\nu$.}
	\label{tab::RD_SDG_Vis}
	\hspace{-15mm}
	\begin{tabular}{ccccc}
	\begin{tabular}[c]{@{}c@{}}Method\\ $\epsilon_{trunc}$\end{tabular}	& \begin{tabular}[c]{@{}c@{}}LRPCG\\ 1e-06 (1e-08)\end{tabular} & \begin{tabular}[c]{@{}c@{}}LRPBiCGstab\\ 1e-06 (1e-08)\end{tabular} &  \begin{tabular}[c]{@{}c@{}}LRPQMRCGstab\\ 1e-06 (1e-08)\end{tabular} & \begin{tabular}[c]{@{}c@{}}LRPGMRES\\ 1e-06 (1e-08)\end{tabular} \\ \hline
	    \hline
		$\nu=1$       &                         &                         &                         &    \\
		Ranks         & 17 (44)                 & 20 (51)                 & 20 (42)                 & 22 (39)  \\
		\#iter        & 4 (4)                   & 3 (3)                   & 3 (3)                   & 4 (4)    \\
		CPU           & 54.3 (62.0)             & 68.7 (75.2)             & 87.9 (91.2)             & 56.4 (56.3) \\
		Resi.         & 8.2189e-07 (1.1215e-06) & 9.9896e-06 (9.9897e-06) & 7.3503e-07 (3.6458e-08) & 1.1062e-06 (1.1062e-06) \\
		Memory        & 831.9 (2300.1)          & 978.8 (2495.8)          & 978.8 (2055.4)          & 1076.6 (1908.6)  \\
		\hline
		\hline
		$\nu=10^{-2}$ &                         &                         &                         &    \\
		Ranks         & 21 (60)                 & 26 (60)                 & 25 (59)                 & 23 (39)  \\
		\#iter        & 4 (4)                   & 3 (3)                   & 3 (3)                   & 4 (4)    \\
		CPU           & 52.4 (64.5)             & 65.9 (72.3)             & 89.5 (94.3)             & 52.3 (52.6) \\
		Resi.         & 7.7284e-07 (1.1225e-06) & 9.9906e-06 (9.9918e-06) & 1.6268e-06 (8.4171e-08) & 1.1074e-06 (1.1074e-06) \\
		Memory        & 1027.7 (2936.3)         & 1272.4 (2936.3)         & 1223.4 (2887.3)         & 1125.6 (1908.6)  \\
		\hline
		\hline
		$\nu=10^{-4}$ &                         &                         &                         &    \\
		Ranks         & 30 (60)                 & 32 (60)                 & 32 (60)                 & 28 (47)  \\
		\#iter        & 4 (4)                   & 3 (3)                   & 3 (3)                   & 4 (4)    \\
		CPU           & 52.5 (58.8)             & 69.3 (73.8)             & 86.1 (87.9)             & 57.9 (57.7)  \\
		Resi.         & 1.0719e-06 (1.1205e-06) & 9.9865e-06 (9.9880e-06) & 6.5595e-07 (2.0226e-07) & 1.1075e-06 (1.1075e-06)  \\
		Memory        & 1468.1 (2936.3)         & 1566 (2936.3)           & 1566 (2936.3)           & 1370.3 (2300.1)  \\
		\hline
		\hline	
	\end{tabular}
}
\end{table}

Next, we investigate the convergence behavior of the low--rank variants of  iterative solvers  with  different values of standard deviation  $\kappa$ for varying values of $\nu$ in Figure~\ref{fig:RD_SDG_Convergence}. For relatively large $\kappa$, we observe that LRPBiCGstab  and LRPGMRES yield better convergence behaviour, whereas the LRPCG method does not converge since the dominance of  nonsymmetrical  increases.

\begin{figure}[t]
	\centering	
\includegraphics[width=1\textwidth]{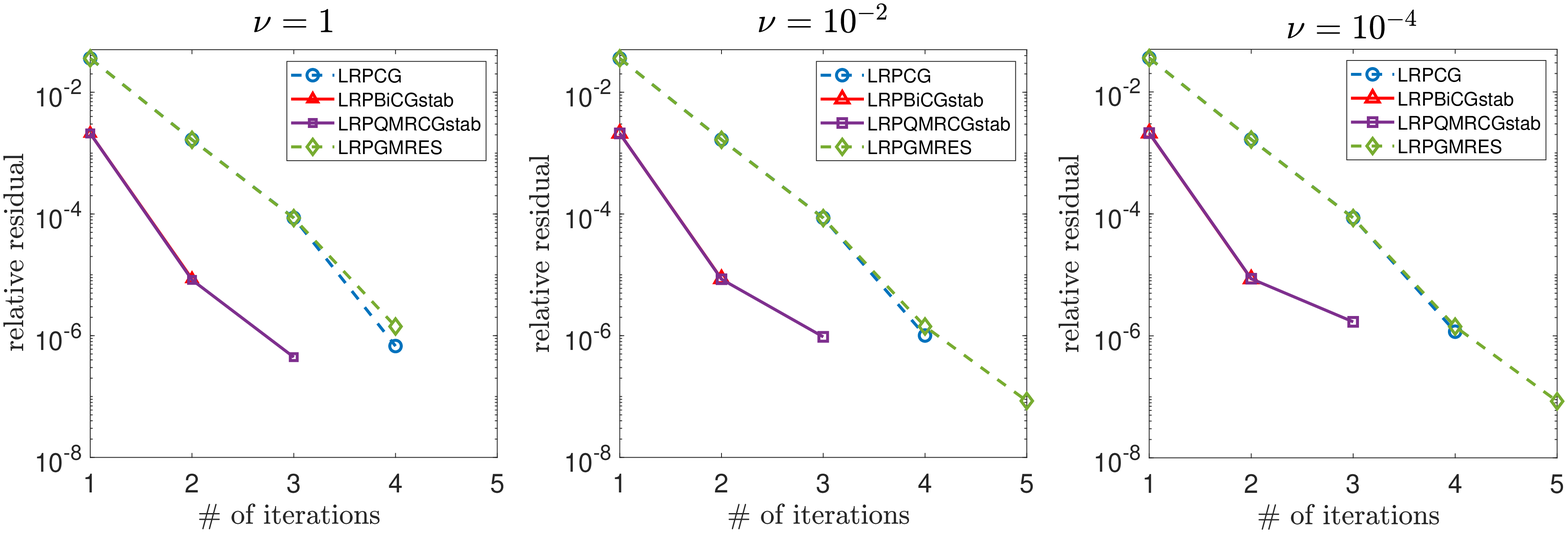}	
\includegraphics[width=1\textwidth]{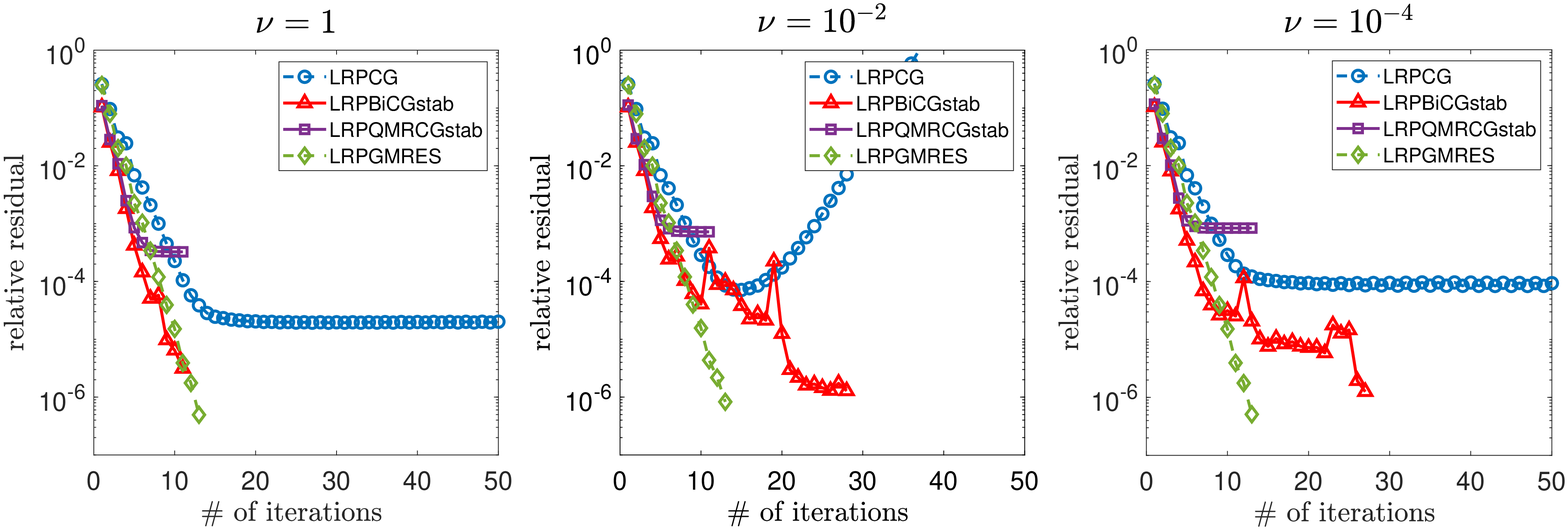}
	\caption{Example~\ref{ex:stationary_diff}: Convergence of low--rank variants of iterative solvers with $\kappa=0.05$ (top) and  $\kappa=0.5$ (bottom) for varying values of viscosity  $\nu$. The mean-based preconditioner $\mathcal{P}_0$ is used with the parameters $N=5$, $Q=3$, $\ell=1$,  $N_d=6144$, and $\epsilon_{trunc}= 10^{-6}$.}
	\label{fig:RD_SDG_Convergence}		
\end{figure}

\begin{table}[H]
	\scriptsize
	\caption{Example~\ref{ex:stationary_diff}: Simulation results showing ranks of truncated solutions, total number of iterations, total CPU times (in seconds),  relative residual, and memory demand of the solution (in KB) with   $N_d=6144 $, $N=7$, $Q=3$, $\ell=1$, $\epsilon_{trunc}= 10^{-6}$, and $\nu=10^{-4}$ for different choices of preconditioners.}
	\label{tab::RD_SDG_KappaPrecond}
    \centering{
	\begin{tabular}{ccccc}
\begin{tabular}[c]{@{}c@{}}Method\\ Preconditioner \end{tabular}			& \begin{tabular}[c]{@{}c@{}}LRPBiCGstab\\ $\mathcal{P}_0$\end{tabular} & \begin{tabular}[c]{@{}c@{}}LRPGMRES\\ $\mathcal{P}_0$\end{tabular} & \begin{tabular}[c]{@{}c@{}}LRPBiCGstab\\ $\mathcal{P}_1$\end{tabular} & \begin{tabular}[c]{@{}c@{}}LRPGMRES\\ $\mathcal{P}_1$\end{tabular} \\ \hline
	    \hline
		$\kappa=0.05$&         &             &               &                               \\
		Ranks  &  32          &  28          & 31            &  27                 \\
		\#iter &  3           &  4           & 3             &  5                 \\
		CPU    &  69.3        &  57.9        & 69.0          &  72.7              \\
		Resi.  &  9.9865e-06  &  1.1075e-06  & 6.0448e-06    &  8.7712e-08        \\
		Memory &  1566        &  1370.3      & 1517.1        &  1321.3            \\
		\hline
		\hline	
		$\kappa=0.5$&           &               &              &                               \\
		Ranks  &  60            &  60           & 60           &  60                  \\
		\#iter &  13            &  13           & 15           &  13                        \\
		CPU    &  781.7         &  248.5        & 913.6        &  245.8                   \\
		Resi.  &  1.2629e-06    &  4.9417e-07   & 1.8697e-06   &  6.9625e-07        \\
		Memory &  2936.3        &  2936.3       & 2936.3       &  2936.3            \\
		\hline
		\hline	
	\end{tabular}}
\end{table}

\begin{figure}[H]
	\centering \includegraphics[width=1\textwidth]{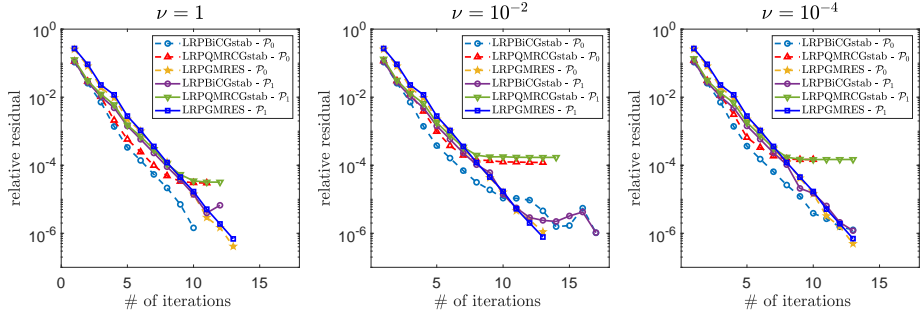}		
	\caption{Example~\ref{ex:stationary_diff}: Convergence of low--rank variants of LRPBiCGstab, LRPQMRCGstab, and LRPGMRES with  $N=7$, $Q=3$, $\ell=1$,  $N_d=6144$, $\epsilon_{trunc}= 10^{-8}$, and  $\kappa=0.5$ for  the mean-based preconditioner $\mathcal{P}_0$ and the Ullmann preconditioner $\mathcal{P}_1$.}
	\label{fig:RD_SDG_Precond}		
\end{figure}

\begin{figure}[htp!]
	\centering
	\includegraphics[width=1\textwidth]{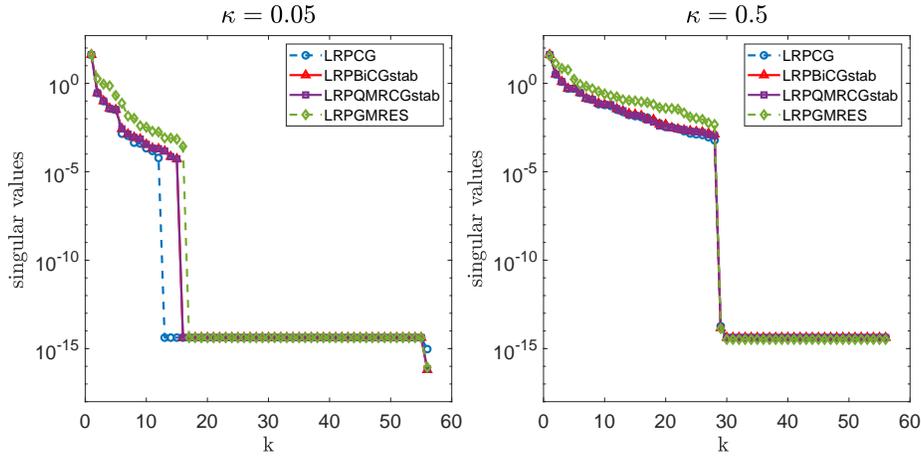}	
	\caption{Example~\ref{ex:stationary_diff}: Decay of singular values of low--rank solution matrix $\mathbf{U}$ obtained by using the mean-based preconditioner $\mathcal{P}_0$ with $N=5$, $Q=3$, $\ell=1$, $N_d=6144$, $\nu=1$, and $\epsilon_{trunc}=10^{-6}$ for $\kappa=0.05$ (left) and $\kappa=0.5$ (right).}
	\label{fig:RD_SDG_singularSoln}		
\end{figure}

In Table~\ref{tab::RD_SDG_KappaPrecond}, we examine the effect of the standard deviation parameter $\kappa$ with $\mathcal{P}_0$  and $\mathcal{P}_1$  preconditioners for only  LRPBiCGstab and LRPGMRES since they exhibit better convergence behaviour; see Figure~\ref{fig:RD_SDG_Convergence}.  As $\kappa$ increases, the low-rank solutions indicate deteriorating performance, regardless of which  the preconditioner or iterative solver are used. 
We also examine the effect of preconditioners on the iterative solvers in Figure~\ref{fig:RD_SDG_Precond} in terms of convergence of iterative solvers. Since LRPCG does not converge for large values of $\kappa$, they are not included. The results show that the mean--based preconditioner $\mathcal{P}_0$ exhibits better convergence behaviour compared to the Ullmann preconditioner $\mathcal{P}_1$ for LRPBiCGstab and  LRPQMRCGstab, whereas they are almost the same for LRPGMRES.

Figure~\ref{fig:RD_SDG_singularSoln}  shows the decay of singular values of low--rank solution matrix $\mathbf{U}$ obtained by using the mean-based preconditioner $\mathcal{P}_0$.  Keeping other parameters fixed, increasing the value of $\kappa$ slows down the decay of the singular values of the obtained solutions. Thus, the total time for solving the system and the time spent on truncation will also increase; see Table~\ref{tab::RD_SDG_KappaPrecond}.

\begin{table}[t]
	\caption{Example~\ref{ex:stationary_diff}: Total CPU times (in seconds) and memory (in KB) for $N_d=6144 $, $Q=3$, $\ell=1$, and $\kappa=0.05$.}
	\label{tab::RD_SDFull}
	\centering
	\begin{tabular}{cccc}
		$\mathcal{A}\backslash \mathcal{F}$ & $\nu=10^{0}$  & $\nu=10^{-2}$ & $\nu=10^{-4}$\\ \hline
		N   & CPU (Memory)             & CPU (Memory)       & CPU (Memory)    \\ \hline
		2   & 10.8 (960)                 & 10.7 (960) 		  & 10.8 (960)      \\
		3   & 1463.7 (1920)           & 1464.2 (1920)      & 1463.7 (1920)   \\
		4   & OoM                             & OoM                & OoM
	\end{tabular}
\end{table}

Last, we display the performance of $\mathcal{A}\backslash \mathcal{F}$ in terms of total CPU times (in seconds) and memory requirements (in KB) in Table~\ref{tab::RD_SDFull}. Some numerical results are not reported since the solution from terminates with "out of memory", which we have denoted as "OoM". A major observation from numerical simulations, low--rank variant of Krylov subspace methods achieve greater computational savings especially in terms of memory.


\subsection{Stationary problem with random convection parameter}\label{ex:stationary_conv}
Our second example is a two-dimensional stationary convection diffusion equation with random velocity. To be precise, we choose  the deterministic diffusion parameter $a(x,\omega)=\nu > 0$, the deterministic source function $f(x)=0$, and the spatial domain $\mathcal{D} = [0,1]^2$. The random velocity field $\mathbf{b}(x,\omega)$ is
\begin{eqnarray}\label{eq:conv}
\mathbf{b}(x,\omega):= \left(  \cos\Big(\frac{1}{5}\eta(x,\omega)\Big), \sin\Big(\frac{1}{5}\eta(x,\omega)\Big) \right)^T,
\end{eqnarray}
\begin{figure}[t]
	\centering
	\includegraphics[width=1\textwidth]{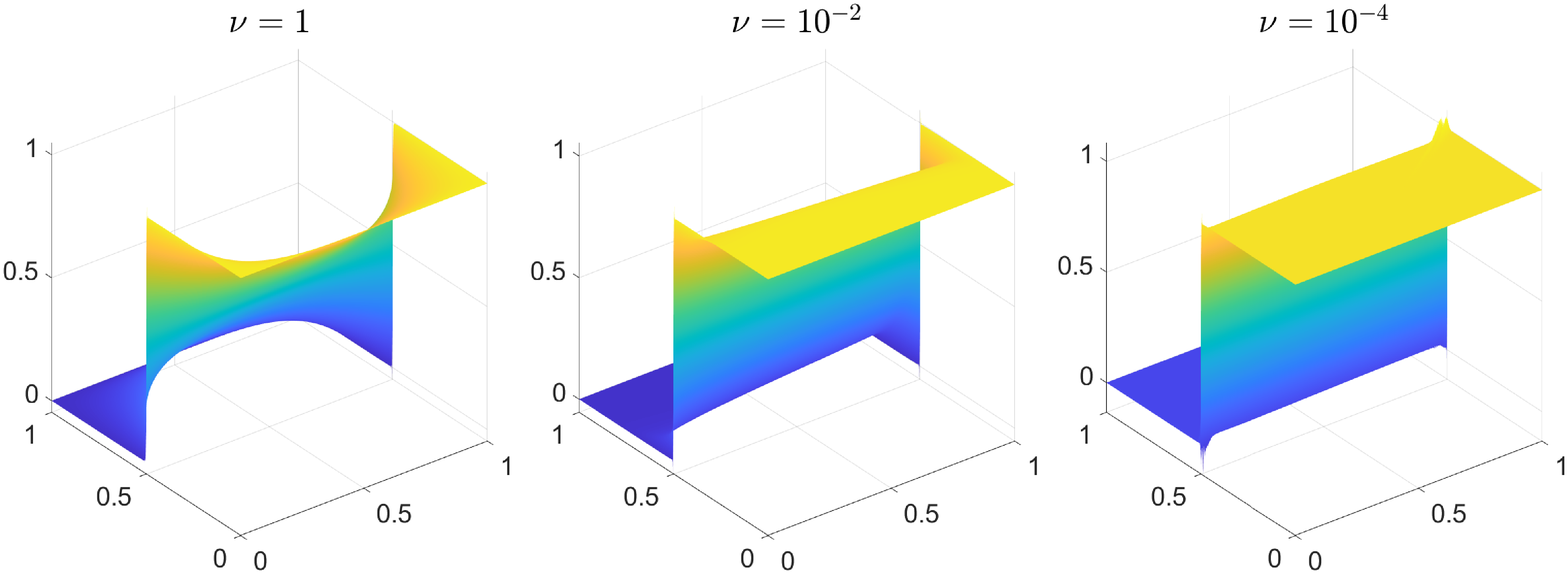}
	\includegraphics[width=1\textwidth]{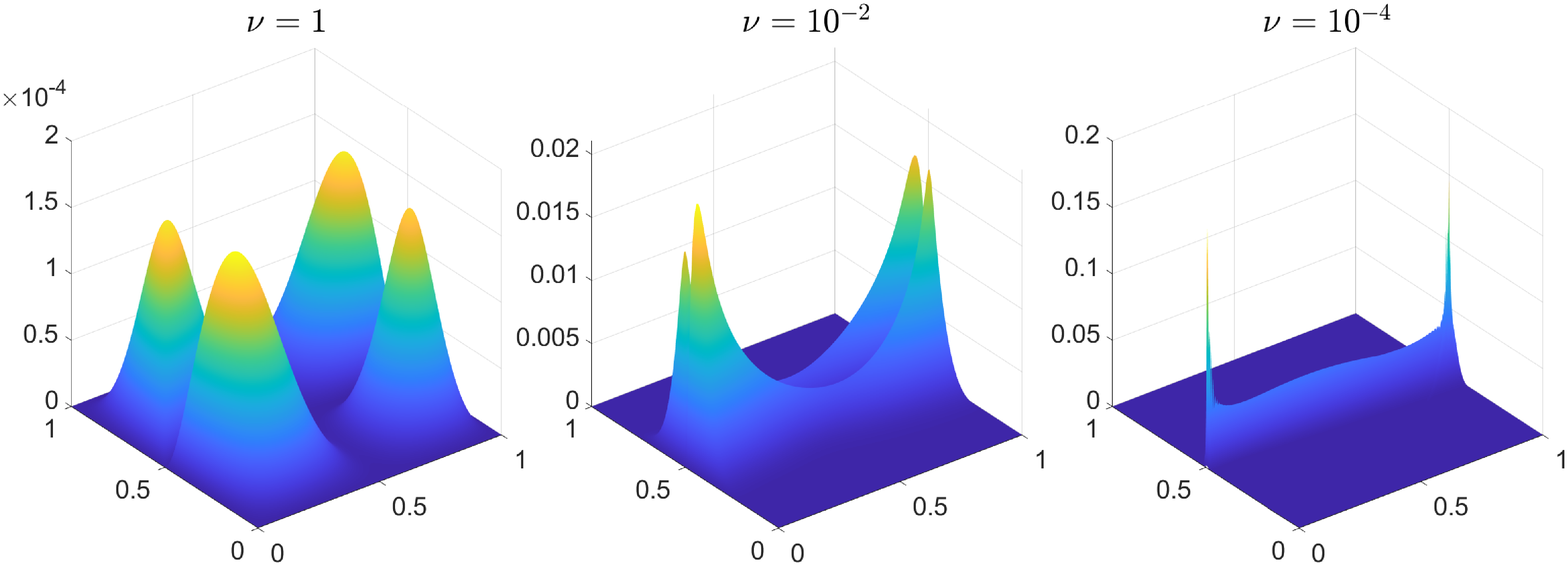}
	\caption{Example~\ref{ex:stationary_conv}: Mean (top) and variance (bottom) of SG solutions obtained by solving $ \mathcal{A}\backslash \mathcal{F} $ with $ N=2 $, $ Q=2 $, $\ell=1$, $N_d=393216$,  and $\kappa=0.05$,  for various values of $\nu$.}
	\label{fig:RC_SCmeanVariance_Full}		
\end{figure}
\noindent where the random field $\eta(x,\omega)$ is chosen as a uniform random field having zero mean with the covariance function defined in \eqref{Cov:Gauss}. The Dirichlet boundary condition $u_d(x)$ is given by
\[
u_d(x)=
        \begin{cases}
        1, & x\in S,\\
        0, & x \in \partial \mathcal{D} \backslash S,
        \end{cases}
\]
where the set $S$ is the subset of $\partial \mathcal{D}$ defined by
\[
\{x_1=0, x_2\in[0,0.5]\}\cup \{x_1\in[0,1], x_2=0\}\cup \{x_1=1, x_2\in[0,0.5]\}.
\]
Due to the random velocity, $\mathbf{b}(x,\omega)$, the solution has  sharp transitions in the domain $\mathcal{D}$ and then spurious oscillations will propagate into the stochastic domain $\Omega$. As $\nu$ decreases, the interior  layer becomes more visible; see Figure~\ref{fig:RC_SCmeanVariance_Full}  for the mean and variance of solution for various values of $\nu$.

\begin{table}[t]
	\scriptsize{
	\caption{Example~\ref{ex:stationary_conv}: Simulation results showing ranks of truncated solutions, total number of iterations, total CPU times (in seconds),  relative residual, and memory demand of the solution (in KB) with $N_d=6144 $, $Q=3$, $\ell=1$, $\kappa=0.05$,  $N=7$, and the mean-based preconditioner $\mathcal{P}_0$ for various values of viscosity parameter $\nu$.}
	\label{tab::RC_SDG_Vis}
	\hspace{-15mm}
	\begin{tabular}{ccccc}
	\begin{tabular}[c]{@{}c@{}}Method\\ $\epsilon_{trunc}$\end{tabular}	& \begin{tabular}[c]{@{}c@{}}LRPCG\\ 1e-06 (1e-08)\end{tabular} & \begin{tabular}[c]{@{}c@{}}LRPBiCGstab\\ 1e-06 (1e-08)\end{tabular}  & \begin{tabular}[c]{@{}c@{}}LRPQMRCGstab\\ 1e-06 (1e-08)\end{tabular} & \begin{tabular}[c]{@{}c@{}}LRPGMRES\\ 1e-06 (1e-08)\end{tabular} \\ \hline
	    \hline
		$\nu=1$    &                         &                         &                         &    \\
		Ranks      & 8 (19)                  & 10 (23)                 & 9 (22)                  & 6 (8)      \\
		\#iter     & 10 (10)                 & 3 (3)                   & 4 (4)                   & 5 (5)      \\
		CPU        & 108.6 (120.4)           & 49.7 (56.0)             & 76.1 (82.4)             & 60.7 (59.7)  \\
		Resi.      & 1.3666e-06 (1.4307e-06) & 5.7868e-07 (5.7870e-07) & 6.8606e-06 (6.8424e-06) & 1.2811e-06 (1.2811e-06)  \\
		Memory     & 391.5 (929.8)           & 489.4 (1125.6)          & 440.4 (1076.6)          & 293.6 (391.5)  \\
		\hline
		\hline
		$\nu=10^{-2}$  &                         &                         &                         &    \\
		Ranks          & 15 (34)                 & 45 (60)                 & 21 (60)                 & 6 (14)      \\
		\#iter         & 100 (100)               & 100 (100)               & 100 (100)               & 100 (100)      \\
		CPU            & 992.2 (1202.8)          & 1782.7 (2133.9)         & 2602.3 (2815.0)         & 8798.7 (8864.7)  \\
		Resi.          & 3.0059e+26 (3.0060e+26) & 1.4807e-01 (2.6669e-02) & 9.0173e-03 (9.3659e-03) & 1.0002e-03 (1.0002e-03)  \\
		Memory         & 734.1 (1663.9)          & 2202.2 (2936.3)         & 1027.7 (2936.3)         & 293.6 (685.1)  \\
		\hline
		\hline
		$\nu=10^{-4}$  &                         &                         &                         &    \\
		Ranks          & 29 (60)                 & 60 (60)                 & 35 (60)                 & 12 (27)      \\
		\#iter         & 100 (100)               & 100 (100)               & 100 (100)               & 100 (100)      \\
		CPU            & 1102.3 (1382.6)         & 2124.4 (2135.5)         & 2802.1 (2869.9)         & 8401.0 (8394.8)  \\
		Resi.          & 8.7243e+26 (8.7242e+26) & 2.2085e-02 (3.2792e-02) & 3.6713e-03 (3.3074e-03) & 1.2075e-03 (1.2075e-03)  \\
		Memory         & 1419.2 (2936.3)         & 2936.3 (2936.3)         & 1712.8 (2936.3)         & 587.3 (1321.3)  \\
		\hline
		\hline	
	\end{tabular}
}
\end{table}

In Table~\ref{tab::RC_SDG_Vis} and ~\ref{tab::RC_SCP0Eps04_Nd}, we display the performance of low--rank of Krylov subspace methods with the mean--based precondition $\mathcal{P}_0$ by considering  various data sets. When $\nu$ decreases, the complexity of the problem  increases in terms of the rank of the truncated solutions, total CPU times (in seconds), and memory demand of the solution (in KB); see Table~\ref{tab::RC_SDG_Vis}. As the previous example,  LRPCG method does not work well for smaller values of $\nu$, whereas LRPGMRES exhibits better  performance. Next, we investigate the convergence behavior of the low--rank variants of iterative solvers with different values of $\nu$ in Figure~\ref{fig:RC_SCG_Convergence}.  While LRPBiCGstab method exhibits oscillatory behaviour, the relative residuals obtained by LRPQMRCGstab and LRPGMRES  decrease monotonically.

\begin{table}[H]
	\scriptsize
	\caption{Example~\ref{ex:stationary_conv}: Simulation results showing ranks of truncated solutions, total number of iterations, total CPU times (in seconds),  relative residual, and memory demand of the solution (in KB) with $N =7$, $Q=3$, $\ell=1$, $\kappa=0.05$,  $\nu=10^{-4}$ and  the mean-based preconditioner $\mathcal{P}_0$ for various values of $N_d$.}
	\label{tab::RC_SCP0Eps04_Nd}
	\hspace{-15mm}
	\begin{tabular}{ccccc}
		& \begin{tabular}[c]{@{}c@{}}LRPCG\\ 1e-06 (1e-08)\end{tabular} & \begin{tabular}[c]{@{}c@{}}LRPBiCGstab\\ 1e-06 (1e-08)\end{tabular} & \begin{tabular}[c]{@{}c@{}}LRPQMRCGstab\\ 1e-06 (1e-08)\end{tabular} & \begin{tabular}[c]{@{}c@{}}LRPGMRES\\ 1e-06 (1e-08)\end{tabular} \\ \hline
		\hline
		$ N_d=384 $ &                         &                         &                         &    \\
		Ranks       & 17 (37)                 & 58 (60)                 & 21 (60)                 & 26 (42)      \\
		\#iter      & 100 (100)               & 100 (100)               & 100 (100)               & 100 (100)    \\
		CPU         & 213.3 (205.7)           & 329.8 (329.4)           & 487.6 (481.7)           & 2119.7 (2135.3) \\
		Resi.       & 1.1637e+28 (1.1637e+28) & 1.5964e-02 (3.3908e-01) & 1.0163e-02 (1.0857e-02) & 7.4718e-06 (7.2284e-06)  \\
		Memory      & 66.9 (145.7)            & 228.4 (236.3)           & 82.7 (236.3)            &  102.4 (165.4)  \\
		\hline
		\hline
		$ N_d=1536$ &                         &                         &                         &    \\
		Ranks       & 25 (56)                 & 60 (60)                 & 21 (55)                 & 30 (49)         \\
		\#iter      & 100 (100)               & 100 (100)               & 100 (100)               & 65 (65)   \\
		CPU         & 305.4 (338.8)           & 497.8 (503.0)           & 699.4 (709.0)           & 1278.3 (1286.1) \\
		Resi.       & 2.9340e+27 (2.9339e+27) & 5.3887e-02 (1.9615e-02) & 6.8316e-03 (7.0652e-03) & 1.8606e-06 (1.8606e-06)  \\
		Memory      & 323.4 (724.5)           & 776.3 (776.3)           & 271.7 (711.6)           & 388.1 (633.9)  \\
		\hline
		\hline
		$ N_d=6144$ &                         &                         &                         &    \\
		Ranks       & 29 (60)                 & 60 (60)                 & 35 (60)                 & 12 (27)      \\
		\#iter      & 100 (100)               & 100 (100)               & 100 (100)               & 100 (100)      \\
		CPU         & 1102.3 (1382.6)         & 2124.4 (2135.5)         & 2802.1 (2869.9)         & 8401.0 (8394.8)  \\
		Resi.       & 8.7243e+26 (8.7242e+26) & 2.2085e-02 (3.2792e-02) & 3.6713e-03 (3.3074e-03) & 1.2075e-03 (1.2075e-03)  \\
		Memory      & 1419.2 (2936.3)         & 2936.3 (2936.3)         & 1712.8 (2936.3)         & 587.3 (1321.3)  \\
		\hline
		\hline
		$ N_d=24576$  &                         &                         &                         &    \\
		Ranks         & 25 (60)                 & 60 (60)                 & 23 (60)                 & 7 (19)      \\
		\#iter        & 100 (100)               & 100 (100)               & 100 (100)               & 100 (100)      \\
		CPU           & 7276.2 (10498.6)        & 16726.6 (16936.7)       & 19960.3 (20646.5)       & 41929.8 (41778.9) \\
		Resi.         & 3.5040e+26 (3.5100e+26) & 5.8952e-03 (1.7120e-02) & 1.7710e-03 (1.6015e-03) & 7.3550e-04 (7.3550e-04)   \\
		Memory        & 4823.4 (11576.3)        & 11576.3 (11576.3)       & 4437.6 (11576.2)        & 1350.6 (3665.8)  \\
		\hline
		\hline
	\end{tabular}
\end{table}

Figure~\ref{fig:RC_SC_Sing1} shows the decay of singular values of low--rank solution matrix $\mathbf{U}$ obtained by using  $\mathcal{P}_0$ and  $\mathcal{P}_1$ preconditioners.  Keeping other parameters fixed, decreasing  the value of $\nu$ slows down the decay of the singular values of the obtained solutions. Thus, the total time for solving the system and the time spent on truncation will  increase; see Table~\ref{tab::RC_SDG_Vis}. In practical applications, one is usually more interested in large--scale simulations in which the degree of freedom (Dof) is quite large. In Table~\ref{tab::RC_GMRES}, we look for memory demand of the solution (in KB) obtained full--rank and low--rank variants of GMRES solver. As expected, low--rank approximation significantly reduces computer memory required to solve the large system.

\begin{figure}[t]
	\centering
	\includegraphics[width=1\textwidth]{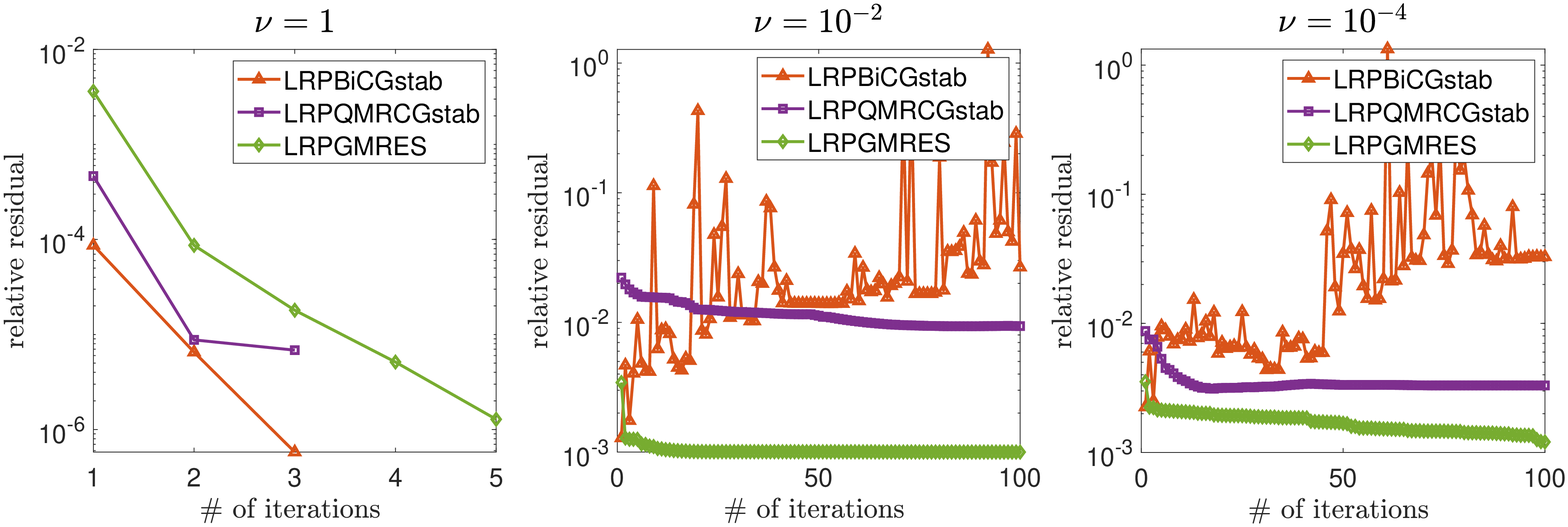}
	\caption{Example~\ref{ex:stationary_conv}: Convergence of low--rank variants of iterative solvers for varying values of viscosity  $\nu$. The mean-based preconditioner $\mathcal{P}_0$ is used with the parameters $N=7$, $Q=3$, $\ell=1$,  $\kappa=0.05$,  $N_d=6144$, and $\epsilon_{trunc}= 10^{-8}$.}
	\label{fig:RC_SCG_Convergence}		
\end{figure}

\begin{figure}[H]
	\centering
	\includegraphics[width=0.95\textwidth]{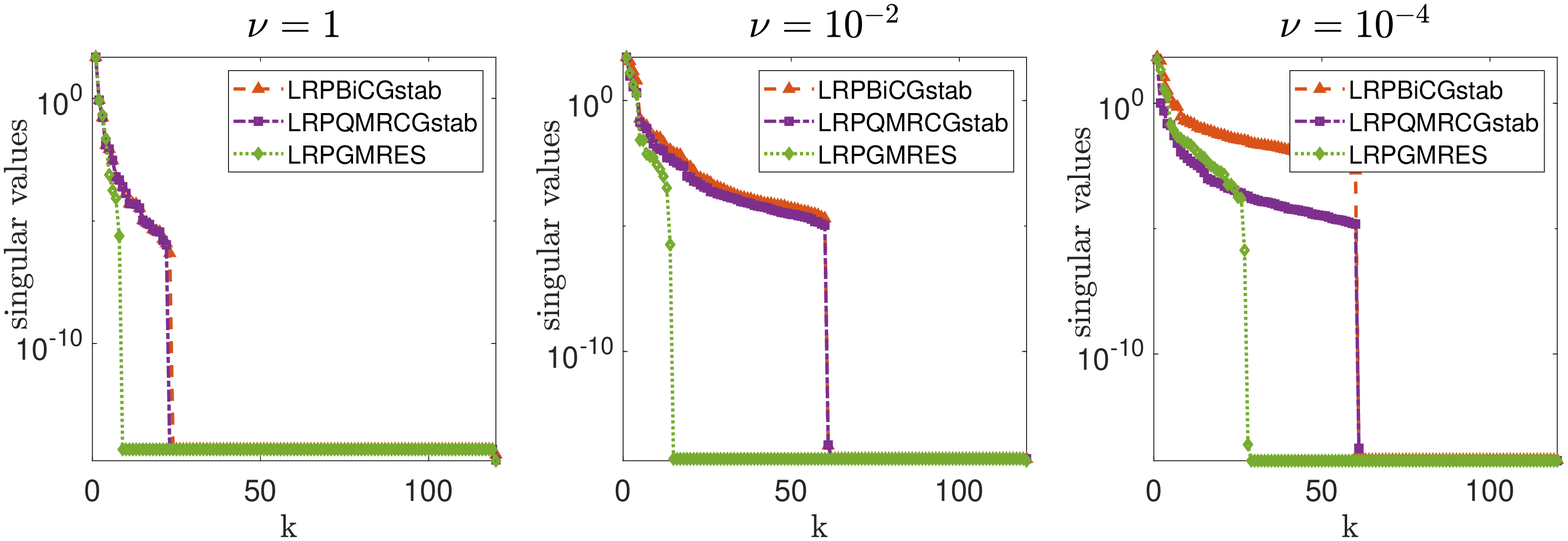}
	\includegraphics[width=0.95\textwidth]{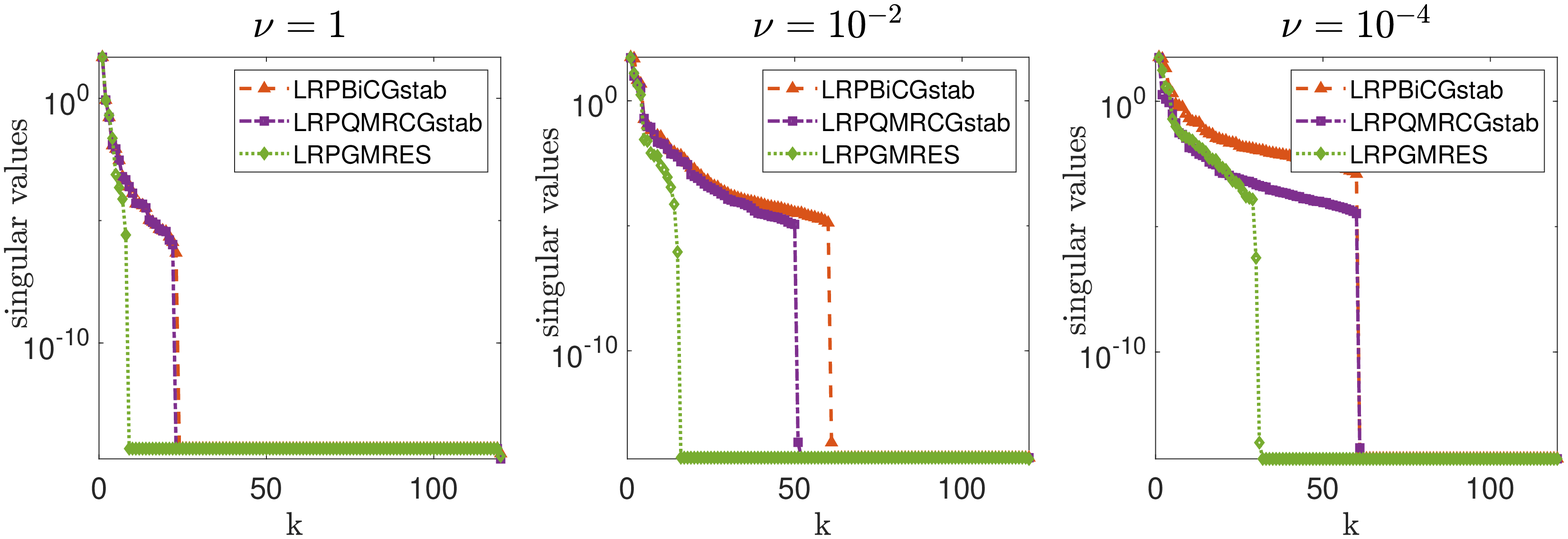}
	\caption{Example~\ref{ex:stationary_conv}: Decay of singular values of solution matrix $ \mathbf{U}  $ with $ N=7 $, $ Q=3 $, $\ell=1$, $N_d=6144$, $\kappa=0.05$, and $\epsilon_{trunc}=10^{-8}$ with the mean-based preconditioner (top) and the Ullmann preconditioner (bottom) for various values of $\nu$.}
	\label{fig:RC_SC_Sing1}		
\end{figure}

\begin{table}[H]
	 \centering
	\scriptsize{
	\caption{Example~\ref{ex:stationary_conv}: Memory demand of the solution (in KB) obtained full--rank and low--rank variants of GMRES solver with $N=7$, $Q=3$, $\ell=1$, $\kappa=0.05$, $\epsilon_{trunc}=10^{-6}$ ($\epsilon_{trunc}=10^{-8}$), and the mean-based preconditioner $\mathcal{P}_0$ for various values of degree of freedom (DoF).}
	\label{tab::RC_GMRES}
	\begin{tabular}{ccccc}
	 DoF  	& 46080 & 184320 & 737280 & 2949120   \\ \hline
	    \hline
		Low--Rank      & 94.5 (157.5) & 323.4 (556.3) & 587.3 (1468.1) &  1543.5 (3665.8)       \\
		Full--Rank     & 360          & 1440          & 5760           &  23040            \\
		\hline
		\hline
	\end{tabular}
}
\end{table}


\subsection{Unsteady problem with random diffusion parameter}\label{ex:unsteady_diff}
Last, we consider an unsteady convection diffusion equation with random diffusion parameter defined on $\mathcal{D}=[0,1]^2$. The rest of data is as follows
\[
T=0.5, \quad \mathbf{b}(x)=(1,1)^T, \quad f(x,t) =0, \quad u^0(x) =0
\]
with the Dirichlet boundary condition
\[
u_d(x)=\begin{cases}
u_d(0,x_2)=x_2(1-x_2), & u_d(1,x_2)=0,\\
u_d(x_1,0)=0, & u_d(x_1,1)=0.
\end{cases}
\]

\begin{figure}[H]
	\centering
	\includegraphics[width=1\textwidth]{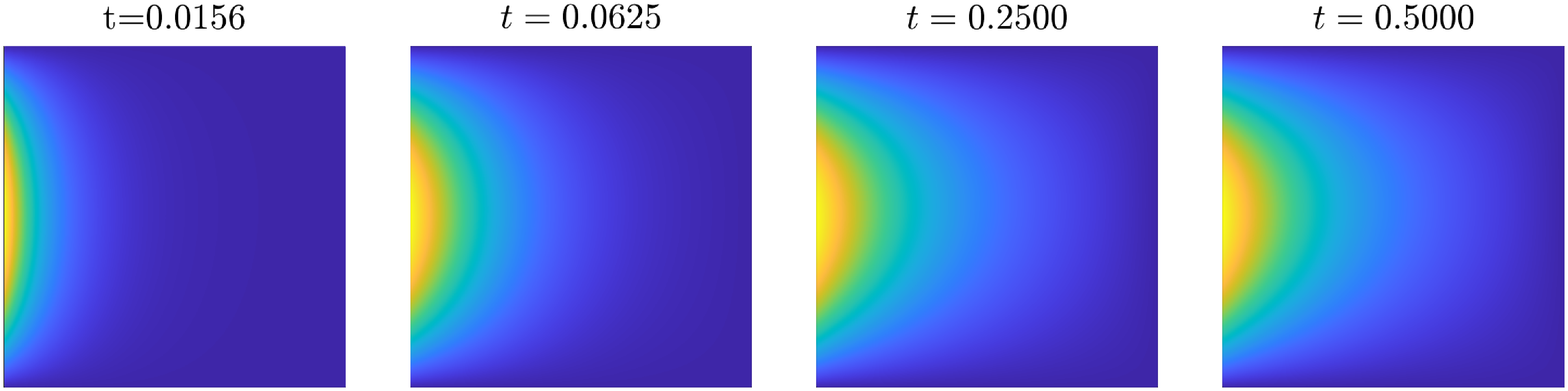}
	\includegraphics[width=1\textwidth]{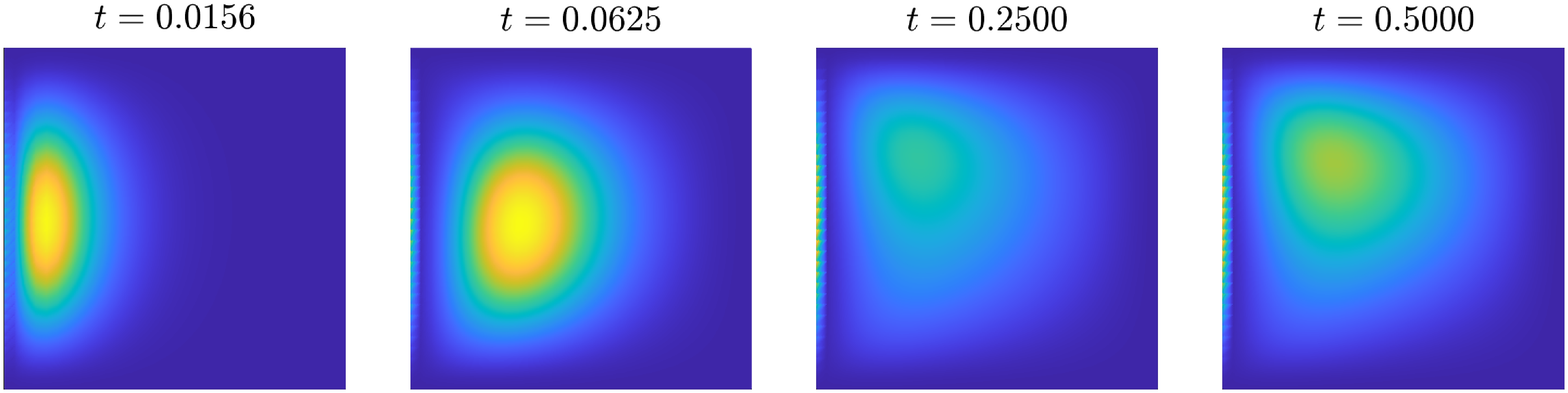}
	
	\caption{Example~\ref{ex:unsteady_diff}: Mean and variance of computed solution at various time steps obtained by  LRPBiCGstab with $N=17$, $Q=3$, $\ell=1.5$, $\sigma=0.15$, $\epsilon_{trunc}=10^{-6}$, and $\mathcal{P}_0$.}
	\label{fig:UDmean_cor1.5}		
\end{figure}

\begin{table}[H]
	\scriptsize{
		\caption{Example~\ref{ex:unsteady_diff}: Simulation results showing ranks of truncated solutions, total number of iterations, total CPU times (in seconds),  relative residual, and memory demand of the solution (in KB) with $N_d=6144 $, $Q=3$, $\kappa=0.15$, and the mean-based preconditioner $\mathcal{P}_0$ for various values of correlation length $\ell$ at final time $T=0.5$.}
		\label{tab::UDG0_corr}
		\hspace{-15mm}
	\begin{tabular}{ccccc}
	\begin{tabular}[c]{@{}c@{}}Method\\ $\epsilon_{trunc}$\end{tabular}	& \begin{tabular}[c]{@{}c@{}}LRPCG\\ 1e-06 (1e-08)\end{tabular} & \begin{tabular}[c]{@{}c@{}}LRPBiCGstab\\ 1e-06 (1e-08)\end{tabular} &  \begin{tabular}[c]{@{}c@{}}LRPQMRCGstab\\ 1e-06 (1e-08)\end{tabular} & \begin{tabular}[c]{@{}c@{}}LRPGMRES\\ 1e-06 (1e-08)\end{tabular} \\ \hline
	\hline
	$\ell=3$, $N=9$    &                         &                         &                         &    \\
	Ranks              & 25 (56)                 & 25 (55)                 & 22 (52)                 & 25 (38)  \\
	\#iter             & 4 (4)                   & 3 (3)                   & 4 (4)                   & 4 (4)    \\
	CPU                & 5348.9 (6543.5)         & 5072.2 (6869.4)         & 9459.5 (11060.9)        & 5974.5 (6005.6) \\
	Resi.              & 2.7590e-04 (2.7601e-04) & 1.2268e-03 (1.2268e-03) & 8.5901e-05 (8.5657e-05) & 2.3936e-04 (2.3936e-04) \\
	Memory             & 1243 (2784.3)           & 1243 (2734.5)           & 1093.8 (2585.4)         & 1243 (1889.3)  \\
	\hline
	\hline
	$\ell=2.5$, $N=10$ &                         &                         &                         &    \\
	Ranks              & 27 (61)                 & 27 (60)                 & 25 (55)                 & 27 (43)  \\
	\#iter             & 4 (4)                   & 3 (3)                   & 4 (4)                   & 4 (4)    \\
	CPU                & 7564.4 (9310.7)         & 7030.5 (9558.5)         & 13158.6 (15636.3)       & 9219.1 (9158.8)  \\
	Resi.              & 2.7397e-04 (2.7405e-04) & 1.2665e-03 (1.2665e-03) & 9.5085e-05 (9.4831e-05) & 2.3810e-04 (2.3810e-04) \\
	Memory             & 1356.3 (3064.3)         & 1356.3 (3014.1)         & 1255.7 (2762.9)         & 1356.3 (2160.1)  \\
	\hline
	\hline
	$\ell=2$, $N=13$   &                         &                          &                         &    \\
	Ranks              & 28 (68)                 & 29 (66)                  & 27 (63)                 & 32 (52)  \\
	\#iter             & 4 (4)                   & 3 (3)                    & 4 (4)                   & 4 (4)    \\
	CPU                & 10813.5 (15435.4)       & 10745.0 (15709.9)        & 18748.2 (24049.6)       & 18496.9 (18284.7)  \\
	Resi.              & 2.6686e-04 (2.6703e-04) & 1.2994e-03 (1.2994e-03)  & 1.0372e-04 (1.0345e-04) & 2.4580e-04 (2.4580e-04) \\
	Memory             & 1466.5 (3561.5)         & 1518.9 (3456.8)          & 1414.1 (3299.6)         & 1676 (2723.5)  \\
	\hline
	\hline
	$\ell=1.5$, $N=17$ &                         &                         &                         &    \\
	Ranks              & 32 (78)                 & 33 (77)                 & 31 (73)                 & 38 (62)  \\
	\#iter             & 4 (4)                   & 3 (3)                   & 4 (4)                   & 4 (4)    \\
	CPU                & 20658.3 (36422.4)       & 22889.2 (33851.9)       & 36876.4 (50963.2)       & 58974.4 (57828.0) \\
	Resi.              & 2.3545e-04 (2.3548e-04) & 1.3217e-03 (1.3217e-03) & 1.0444e-04 (1.0425e-04) & 2.9531e-04 (2.9531e-04) \\
	Memory             & 1821 (4438.7)           & 1877.9 (4381.8)         & 1764.1 (4154.2)         & 2162.4 (3528.2)  \\
	\hline
	\hline
\end{tabular}
	}
\end{table}

The random diffusion coefficient $a(x,w)$ is a uniform random field having unity mean with the covariance function (\ref{Cov:Gauss}). In the numerical simulations, the number of time points is chosen as $N_T = 32$. From literature, see, e.g., [33], we know that decreasing the correlation length slows down the decay of the eigenvalues in the KL expansion of the random variable $a(\mathbf{x},\omega)$ and therefore, more random variables are  required to sufficiently  capture the randomness. That is, it results in  an increase in the truncation parameter $N$: The reverse is the case when the correlation length is increased. Therefore, the effect of correction length on the low--rank variants of the iterative solver  is our main focus for this benchmark problem. With the help of the following computation as done in \cite{HCElman_TSu_2018},
\begin{equation*}
  \left( \sum \limits_{i=1}^{N} \lambda_i \right) / \left( \sum \limits_{i=1}^{M_{\ell}} \lambda_i \right) > 0.97,
\end{equation*}
we can compute suitable truncation number $N$ for the given correlation length $\ell$. Here, $M_{\ell}$ is a large number which we set $1000$. Computed mean and variance of the solution are displayed in Figure~\ref{fig:UDmean_cor1.5} for various time steps.

\begin{table}[htp!]
   \centering
	\scriptsize{
		\caption{Example~\ref{ex:unsteady_diff}: Simulation results showing total number of iterations, total CPU times (in seconds), and memory demand of the full--rank solution (in KB) with $N_d=6144 $, $Q=3$, $\kappa=0.15$, and the mean-based preconditioner $\mathcal{P}_0$ for various values of correlation length $\ell$ at final time $T=0.5$.}
		\label{tab::UDG0_corr_full}
		\begin{tabular}{cccc}
			\begin{tabular}[c]{@{}c@{}}Method\\ $\epsilon_{tol}$\end{tabular}	& \begin{tabular}[c]{@{}c@{}} PCG\\ 1e-04 \end{tabular} &
			\begin{tabular}[c]{@{}c@{}} PBiCGstab\\ 1e-04 \end{tabular} &
			\begin{tabular}[c]{@{}c@{}} PGMRES\\ 1e-04 \end{tabular} \\ \hline
			\hline
			$\ell=3$, $N=9$    &    &   &        \\
			\#iter & 11  & 5.5   & 10       \\
			CPU    & 7910.7 & 7863.4  & 8727.0   \\
			Memory & 10560 & 10560 & 10560   \\
			\hline
			\hline
		    $\ell=2$, $N=13$    &    &   &        \\
		    \#iter & 11  & 5.5   & 10       \\
		    CPU    & 20252.0 & 20148.5 & 22318.8   \\
		    Memory & 26880 & 26880 & 26880   \\
		    \hline
		    \hline
		    $\ell=1.5$, $N=17$    &    &   &        \\
		    \#iter & 11  & 5.5   & 10       \\
		    CPU    & 41345.0 & 41207.4 & 45499.8   \\
		    Memory & 54720 & 54720 & 54720   \\
		    \hline
		    \hline
		\end{tabular}
	}
\end{table}

\begin{figure}[htp!]
	\centering
	\includegraphics[width=0.5\textwidth]{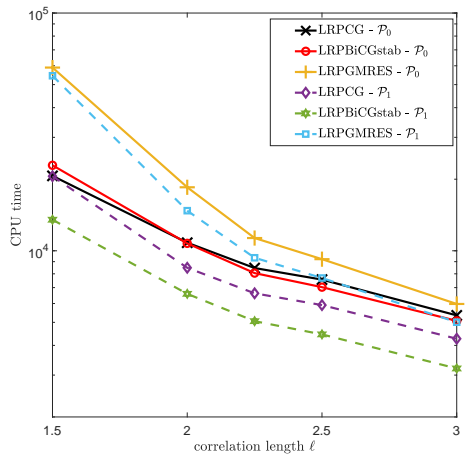}
	\caption{Example~\ref{ex:unsteady_diff}: CPU times of LRPCG, LRPBiCGstab, and LRPGMRES iterative solvers obtained by the preconditioners $\mathcal{P}_0$ and $\mathcal{P}_1$ with $Q=3$, $N_d=6144$, and $\kappa=0.15$ for various values of correlation length $\ell$.}
	\label{fig:cpu_corr}	
\end{figure}

Table~\ref{tab::UDG0_corr} displays the results of numerical simulations for the mean-based preconditioner $\mathcal{P}_0$ for varying values of the correlation length $\ell$. Provided that $97\%$  of the total variance is captured, the small correlation length increases the rank of the computed low--rank solutions and the number of iterations regardless of which the iterative solver is used. Another observation is that decreasing the truncation tolerance $\epsilon_{trunc}$ does not affect the relative residuals but, as expected, at the cost of comparatively more computational time and memory requirements. Next, numerical results obtained by using the standard Krylov subspace iterative solvers  with  the mean-based preconditioner $\mathcal{P}_0$ are displayed in Table~\ref{tab::UDG0_corr_full}. Compared to the full--rank solvers in Table~\ref{tab::UDG0_corr_full}, low--rank Krylov subspace solvers generally  exhibit better performance; see Table~\ref{tab::UDG0_corr}. Regarding of the preconditioners,  Ullmann preconditioner $\mathcal{P}_1$  produce better performance in terms of computational time; see Figure~\ref{fig:cpu_corr}.


\section{Conclusions}	\label{sec:conclusions}
	
In this paper, we have numerically studied  the statistical moments of a convection diffusion equation having random coefficients.  With the help of the stochastic Galerkin approach, we transform the original problem into a system consisting of deterministic convection diffusion equations for each realization of
random coefficients. Then, the symmetric interior penalty Galerkin method is used
to discretize the deterministic problems due to its local mass conservativity. To reduce computational time and memory requirements, we have used low--rank variants of various Krylov subspace methods, such as CG, BiCGstab, QMRCGstab, and  GMRES with suitable preconditioners. It has been  shown in the numerical simulations that LRPGMRES exhibits better performance, especially for convection dominated models.


\section*{Acknowledgements}
This work was supported by TUBITAK 1001 Scientific and Technological Research Projects Funding Program with project number 119F022.




\appendix

\section{Proof of Theorem~\ref{thm:bestapp}}\label{app:A}
\noindent By choosing $\widetilde{v}=\Pi_n(\mathcal{R}_h(v))$, we obtain
\begin{eqnarray*}\label{err1}
\lVert v- \widetilde{v} \rVert_{L^2(H^1(\mathcal{D});\Gamma)}\hspace{-3mm}&\leq& \hspace{-3mm}\lVert v- \mathcal{R}_h(v) \rVert_{L^2(H^1(\mathcal{D});\Gamma)} + \lVert \mathcal{R}_h(v)- \Pi_n(\mathcal{R}_h(v)) \rVert_{L^2(H^1(\mathcal{D});\Gamma)} \nonumber \\
&=&\hspace{-3mm}\lVert v- \mathcal{R}_h(v) \rVert_{L^2(H^1(\mathcal{D});\Gamma)} + \lVert \mathcal{R}_h(v- \Pi_n(v)) \rVert_{L^2(H^1(\mathcal{D});\Gamma)}
\end{eqnarray*}
for a fixed $v\in L^2(H^2(\mathcal{D});\Gamma)\cap H^{q+1}(H^1(\mathcal{D});\Gamma)$. In view of the estimate in \eqref{approxfm}, we have
\begin{eqnarray}\label{ineq:Rv}
\lVert v- \mathcal{R}_h(v) \rVert_{L^2(H^1(\mathcal{D});\Gamma)} \leq C h^{\min(\ell+1,2)-1} \lVert v \rVert_{L^2(H^2(\mathcal{D});\Gamma)}.
\end{eqnarray}
With the help of the  $H^1$--projection operator in \eqref{eq:Hprojc}, the  Cauchy--Schwarz inequality, the $L^2$--projection operator in \eqref{eq:l2projc}, and the approximation in \eqref{estr},  we obtain
\begin{eqnarray}\label{ineq:Rvt}
\lVert \mathcal{R}_h(v- \Pi_n(v)) \rVert_{L^2(H^1(\mathcal{D});\Gamma)} \hspace{-2mm}&\leq&\hspace{-2mm} C \lVert v- \Pi_n(v) \rVert_{L^2(H^1(\mathcal{D});\Gamma)} \nonumber \\
\hspace{-2mm}&\leq& \hspace{-2mm}C \sum_{n=1}^{N}\bigg(\frac{k_n}{2}\bigg)^{q_n+1}\dfrac{\lVert \partial^{q_n+1}_{\xi_n}v \rVert_{L^2(H^1(\mathcal{D});\Gamma)} }{(q_n+1)!}.
\end{eqnarray}
Combining (\ref{ineq:Rv}) and (\ref{ineq:Rvt}), we get
\begin{eqnarray*}
	\lVert v- \widetilde{v} \rVert_{L^2(H^1(\mathcal{D});\Gamma)}
	&\leq& Ch^{\min(\ell+1,2)-1} \lVert v \rVert_{L^2(H^2(\mathcal{D});\Gamma)} \\
&& + \sum_{n=1}^{N}\bigg(\frac{k_n}{2}\bigg)^{q_n+1}\dfrac{\lVert \partial^{q_n+1}_{\xi_n}v \rVert_{L^2(H^1(\mathcal{D});\Gamma)} }{(q_n+1)!},
\end{eqnarray*}
which implies \eqref{ineq:grad}.
\noindent For the derivation of \eqref{ineq:grad2}, we follow the same strategy:
\begin{align*}
 &	\lVert \nabla^2( v- \widetilde{v}) \rVert_{L^2(L^2(\mathcal{D});\Gamma)} 	\nonumber \\
 & \qquad \leq \lVert \nabla^2 (v- \mathcal{R}_h(v)) \rVert_{L^2(L^2(\mathcal{D});\Gamma)} + \lVert \nabla^2 (\mathcal{R}_h(v)- \Pi_n(\mathcal{R}_h(v))\rVert_{L^2(L^2(\mathcal{D});\Gamma)} \\
 & \qquad = \lVert \nabla^2 ( v- \mathcal{R}_h(v)) \rVert_{L^2(L^2(\mathcal{D});\Gamma)} + \lVert \nabla^2 (\mathcal{R}_h(v- \Pi_n(v)) \rVert_{L^2(L^2(\mathcal{D});\Gamma)}.
\end{align*}
An application of the inverse inequality \eqref{eq:inv} on $\mathcal{R}_h(v)$, the definition of $H^1$--projection operator  \eqref{eq:Hprojc}, and the Cauchy--Schwarz inequality yields
\begin{eqnarray}\label{err2}
\lVert \nabla^2(\mathcal{R}_h(v)) \rVert_{L^2(\mathcal{D})} \leq Ch^{-1}\lVert \nabla(\mathcal{R}_h(v)) \rVert_{L^2(\mathcal{D})} \leq Ch^{-1}\lVert \nabla v \rVert_{L^2(\mathcal{D})}.
\end{eqnarray}
By \eqref{estr}, \eqref{err2}, and \eqref{approxfm},
\begin{eqnarray*}
	\lVert \nabla^2( v- \tilde{v}) \rVert_{L^2(L^2(\mathcal{D});\Gamma)}
	&\leq& \lVert \nabla^2 ( v- \mathcal{R}_h(v)) \rVert_{L^2(L^2(\mathcal{D});\Gamma)} \\
	&& + Ch^{-1}\lVert \nabla (v- \Pi_n(v)) \rVert_{L^2(\mathcal{D};\Gamma)}\\
	&\leq& C h^{\min(\ell+1,2)-2} \lVert v \rVert_{L^2(H^2(\mathcal{D});\Gamma)} \\
    && + Ch^{-1}\sum_{n=1}^{N}\bigg(\frac{k_n}{2}\bigg)^{q_n+1}\dfrac{\lVert \partial^{q_n+1}_{\xi_n}v \rVert_{L^2(H^1(\mathcal{D});\Gamma)} }{(q_n+1)!},
\end{eqnarray*}
which is the desired result.

\section{Proof of Theorem~\ref{thm:main}}\label{app:B}
\noindent Decompose $\lVert u-u_h\rVert_{\xi}$ as
\begin{eqnarray}\label{err3}
	\lVert u-u_h\rVert_{\xi} \leq \lVert u_h-\widetilde{u}\rVert_{\xi} +  \lVert u-\widetilde{u}\rVert_{\xi},
\end{eqnarray}
where $\widetilde{u} \in V_h  \otimes \mathcal{Y}_k^q $ is an approximation of the solution $u$, satisfying the Theorem~\ref{thm:bestapp}.

Now, we first find a bound for the first term of \eqref{err3}. By the coercivity of the bilinear form \eqref{coer}, the Galerkin orthogonality, an integration by parts over convective term in the bilinear form \eqref{eq:bilinear}, and the assumption on the convective term $\nabla \cdot \mathbf{b} =0$, we obtain
\begin{eqnarray}\label{err4}
   c_{cv} \lVert u_h-\widetilde{u}\rVert_{\xi}^2 \hspace{-3mm}&\leq& a_{\xi}(u_h-\widetilde{u},u_h-\tilde{u}) \nonumber \\
 	\hspace{-3mm}&=& a_{\xi}(u-\widetilde{u},\underbrace{u_h-\tilde{u}}_{\chi \in V_h}) +\underbrace{a_{\xi}(u-u_h,u_h-\widetilde{u})}_{=0} \nonumber \\
 	\hspace{-3mm}&=&\hspace{-4mm} \int \limits_{\Gamma}\hspace{-1.5mm}\rho(\xi) \Bigg[  \sum \limits_{K \in \mathcal{T}_h} \int \limits_{K} \hspace{-1mm}a\nabla (u-\widetilde{u}) \cdot  \nabla \chi \; dx
 	- \hspace{-4mm} \sum \limits_{ E \in \mathcal{E}^{0}_h \cup \mathcal{E}_h^{\partial}} \int \limits_E \average{a \nabla (u-\widetilde{u}) }  \jump{\chi} \; ds \nonumber \\
 	&& \hspace{-3.5mm}-\hspace{-4mm} \sum \limits_{ E \in \mathcal{E}^{0}_h \cup \mathcal{E}_h^{\partial}} \int \limits_E \average{a \nabla \chi }  \jump{u-\widetilde{u}} \; ds
 	+\hspace{-3.5mm} \sum \limits_{ E \in \mathcal{E}^{0}_h \cup \mathcal{E}_h^{\partial}} \frac{\sigma }{h_E} \int \limits_E \jump{u-\widetilde{u}} \cdot \jump{\chi} \; ds   \nonumber \\
 	&& \hspace{-3.5mm}- \hspace{-1.5mm}\sum \limits_{K \in \mathcal{T}_h} \int \limits_{K} \mathbf {b} \cdot  (u-\tilde{u}) \nabla \chi \; dx   - \hspace{-1.5mm}\sum \limits_{K \in \mathcal{T}_h}\; \hspace{-0.5mm}\int \limits_{\partial K^{+} \backslash \partial \mathcal{D}} \hspace{-3.5mm}\mathbf {b} \cdot \mathbf{n_E} (u-\tilde{u})(\chi^e-\chi) \; ds \nonumber \\
 	&&\hspace{-3.5mm} + \sum \limits_{K \in \mathcal{T}_h} \; \int \limits_{\partial K^{+} \cap \mathcal{D}^{+}} \mathbf {b}\cdot \mathbf{n_E} (u-\tilde{u}) \chi  \; ds \Bigg]d\xi \nonumber \\
 	&\leq& \lvert T_1 + T_2 + T_3 + T_4 + T_5 + T_6 + T_7\rvert.
\end{eqnarray}
With the help of the bound on $a(x,\omega)$ \eqref{abound}, Cauchy--Schwarz inequality, Young's inequality, and Theorem~\ref{thm:bestapp}, we obtain the following bound for the first term in \eqref{err4}
\begin{eqnarray*}
\lvert T_1\rvert  &\leq& \int \limits_{\Gamma}\rho \sqrt{a_{\max}} \Bigg( \sum \limits_{K \in \mathcal{T}_h}  \lVert\nabla (u-\widetilde{u})\rVert^2_{L^2(K)} \Bigg)^{\frac{1}{2}}
 \Bigg( \sum \limits_{K \in \mathcal{T}_h} \lVert \sqrt{a_{\max}} \nabla \chi\rVert^2_{L^2(K)} \Bigg)^{\frac{1}{2}} d\xi \nonumber \\
 &\leq& \int \limits_{\Gamma}\rho \Bigg( \frac{2}{c_{cv}}a_{\max}\sum \limits_{K \in \mathcal{T}_h} \lVert\nabla (u-\widetilde{u})\rVert^2_{L^2(K)}
 + \frac{c_{cv}}{8} \lVert \chi \rVert_e^2\Bigg) d\xi \nonumber
\end{eqnarray*}
\begin{eqnarray*}
 &\leq& C \sum \limits_{K \in \mathcal{T}_h}  \lVert\nabla (u-\widetilde{u})\rVert^2_{L^2(L^2(K);\Gamma)}
 + \frac{c_{cv}}{8} \lVert \chi \rVert_{\xi}^2 \nonumber \\
 &\leq& C \Bigg(h^{\min(\ell+1,2)-1} \lVert u \rVert_{L^2(H^2(\mathcal{D});\Gamma)} + \sum_{n=1}^{N}\bigg(\frac{k_n}{2}\bigg)^{q_n+1}\dfrac{\lVert \partial^{q_n+1}_{\xi_n} u \rVert_{L^2(H^1(\mathcal{D});\Gamma)} }{(q_n+1)!}\Bigg) \nonumber \\
 && + \frac{c_{cv}}{8} \lVert \chi \rVert_{\xi}^2.
\end{eqnarray*}
Next, we derive an estimate for the second and third terms in \eqref{err4}. An application of  Cauchy--Schwarz inequality, Young's inequality, the trace inequality \eqref{eq:trace} for $E \in K_1^E \cap K_2^E$, and Theorem~\ref{thm:bestapp} yields
\begin{eqnarray*}
\lvert T_2\rvert \hspace{-3mm}&\leq&\hspace{-3mm} \int \limits_{\Gamma}\rho \Bigg[  \frac{c_{cv}}{8} \sum \limits_{ E \in \mathcal{E}^{0}_h \cup \mathcal{E}_h^{\partial}}\frac{\sigma }{h_E} \lVert  \chi\rVert^2_{L^2(E)}
	+ \frac{2}{c_{cv}} \sum \limits_{ E \in \mathcal{E}^{0}_h \cup \mathcal{E}_h^{\partial}} \frac{h_E}{\sigma }\lVert\average{a \nabla (u-\widetilde{u}) } \rVert^2_{L^2(E)} \Bigg]d\xi \nonumber \\
	\hspace{-3mm}&\leq&\hspace{-3mm}   C\hspace{-1.5mm}\int \limits_{\Gamma}\hspace{-1.5mm}\rho \hspace{-2.5mm} \sum \limits_{ E \in \mathcal{E}^{0}_h \cup \mathcal{E}_h^{\partial}} \hspace{-2.5mm} \frac{h_E}{\sigma} h_E \lvert K_1^E \rvert^{-1} \Bigg( \lVert \nabla(u-\widetilde{u}) \rVert_{L^2(K_1^E)} + h_{K_1^E}\lVert \nabla^2(u-\widetilde{u}) \rVert_{L^2(K_1^E)}\Bigg)^2 \hspace{-1.5mm}d\xi \nonumber \\
	&&\hspace{-3.5mm} + C\hspace{-1.5mm}\int \limits_{\Gamma}\hspace{-1.5mm}\rho \hspace{-2.5mm} \sum \limits_{ E \in \mathcal{E}^{0}_h \cup \mathcal{E}_h^{\partial}} \hspace{-2.5mm} \frac{h_E}{\sigma} h_E \lvert K_2^E \rvert^{-1} \Bigg( \lVert \nabla(u-\widetilde{u}) \rVert_{L^2(K_2^E)} + h_{K_2^E}\lVert \nabla^2(u-\widetilde{u}) \rVert_{L^2(K_2^E)}\Bigg)^2\hspace{-1.5mm} d\xi \nonumber \\
    && \hspace{-3.5mm}+\frac{c_{cv}}{8}\lVert \chi \rVert_{\xi}^2 \nonumber  \\
	&\leq& \hspace{-3.5mm} C \Bigg( \lVert\nabla(u-\widetilde{u}) \rVert_{L^2(L^2(\mathcal{D});\Gamma)} + h\lVert \nabla^2(u-\widetilde{u}) \rVert_{L^2(H_0^1(\mathcal{D});\Gamma)}\Bigg)^2 + \frac{c_{cv}}{8}\lVert \chi \rVert_{\xi}^2\nonumber \\	
	&\leq&\hspace{-3.5mm}  C \Bigg(h^{\min(\ell+1,2)-1} \lVert u \rVert_{L^2(H^2(\mathcal{D});\Gamma)} + \sum_{n=1}^{N}\bigg(\frac{k_n}{2}\bigg)^{q_n+1}\dfrac{\lVert \partial^{q_n+1}_{\xi_n} u \rVert_{L^2(H^1(\mathcal{D});\Gamma)} }{(q_n+1)!}\Bigg)^2 \nonumber \\
    && + \frac{c_{cv}}{8}\lVert \chi \rVert_{\xi}^2, \nonumber \\
\lvert T_3 \rvert \hspace{-3mm}&\leq& \hspace{-3mm}\int \limits_{\Gamma}\rho  \sum \limits_{ E \in \mathcal{E}^{0}_h \cup \mathcal{E}_h^{\partial}} \lVert\average{a \nabla \chi  } \rVert_{L^2(E)} \lVert \jump{u-\widetilde{u}}\rVert_{L^2(E)} \,  d\xi \\
	&\leq& \hspace{-3mm}\int \limits_{\Gamma}\rho  \sum \limits_{K \in \mathcal{T}_h}  \Bigg( C\bigg( \lVert u-\widetilde{u} \rVert_{L^2(K)} + h_{K} \lVert \nabla(u-\widetilde{u}) \rVert_{L^2(K)} \bigg) a \lVert \nabla\chi\rVert_{L^2(K)}\Bigg) \,  d\xi\\
	&\leq& \hspace{-3mm}C\Bigg(h^{\min(\ell+1,2)-1} \lVert u \rVert_{L^2(H^2(\mathcal{D});\Gamma)} + \sum_{n=1}^{N}\bigg(\frac{k_n}{2}\bigg)^{q_n+1}\dfrac{\lVert \partial^{q_n+1}_{\xi_n} u \rVert_{L^2(H^1(\mathcal{D});\Gamma)} }{(q_n+1)!}\Bigg)^2 \\
&&+ \frac{c_{cv}}{8}\lVert \chi\rVert_{\xi}^2.
\end{eqnarray*}
By Cauchy--Schwarz inequality, Young's inequality, the trace inequality \eqref{eq:trace}, and Theorem~\ref{thm:bestapp}, we find an upper bound for $T_4$ in \eqref{err4}
\begin{eqnarray*}
	\lvert T_4\rvert \hspace{-3mm}&\leq& \hspace{-3mm} \int \limits_{\Gamma}\rho  \Bigg[  \sum \limits_{ E \in \mathcal{E}^{0}_h \cup \mathcal{E}_h^{\partial}} \frac{\sigma }{h_E} \int \limits_E \jump{(u-\widetilde{u})} \cdot \jump{\chi} \Bigg]  \, d\xi  \\
    &\leq& \frac{2}{c_{cv}} \int \limits_{\Gamma}\rho  \sum \limits_{ E \in \mathcal{E}^{0}_h \cup \mathcal{E}_h^{\partial}} \bigg( \frac{\sigma }{h_E} \bigg)  \lVert\jump{ u-\widetilde{u}}\rVert^2_{L^2(E)} d\xi\\
    &&+ \frac{c_{cv}}{8} \int \limits_{\Gamma}\rho   \sum \limits_{ E \in \mathcal{E}^{0}_h \cup \mathcal{E}_h^{\partial}} \bigg(\frac{\sigma }{h_E}\bigg) \lVert \jump{\chi} \rVert^2_{L^2(E)} d\xi\\
    &\leq& \frac{2}{c_{cv}} \int \limits_{\Gamma}\rho  \sum \limits_{K \in \mathcal{T}_h}   C\bigg( \lVert u-\widetilde{u} \rVert_{L^2(K)} + h_{K} \lVert \nabla(u-\widetilde{u}) \rVert_{L^2(K)} \bigg)^2  d\xi + \frac{c_{cv}}{8} \lVert\chi \rVert_{\xi}^2\\
    &\leq& C\Bigg(h^{\min(\ell+1,2)-1} \lVert u \rVert_{L^2(H^2(\mathcal{D});\Gamma)} + \sum_{n=1}^{N}\bigg(\frac{k_n}{2}\bigg)^{q_n+1}\dfrac{\lVert \partial^{q_n+1}_{\xi_n} u \rVert_{L^2(H^1(\mathcal{D});\Gamma)} }{(q_n+1)!}\Bigg)^2  \\
    &&+ \frac{c_{cv}}{8}\lVert \chi\rVert_{\xi}^2.
\end{eqnarray*}
Now, we derive estimates for the convective terms in \eqref{err4}. By following the similar steps as done before with $\mathbf{b} \in \big( L^{\infty}(\overline{\mathcal{D}}) \big)^2$, we obtain
\begin{eqnarray*}
	\lvert T_5\rvert\hspace{-3mm} &\leq& \hspace{-3mm}\int \limits_{\Gamma}\rho \Bigg( \frac{2}{c_{cv}} \dfrac{\|\mathbf{b}\|_{L^{\infty}(\mathcal{D})}}{\sqrt{a_{\max}}} \sum \limits_{ K \in \mathcal{T}_h}  \lVert  u-\widetilde{u}\rVert^2_{L^2(K)}
	 + \frac{c_{cv}}{8} \sum \limits_{K \in \mathcal{T}_h} \lVert \sqrt{a_{\max}}\nabla \chi \rVert^2_{L^2(K)}\Bigg) d\xi \\
	 &\leq& \frac{2}{c_{cv}}C \sum \limits_{ K \in \mathcal{T}_h} \int \limits_{\Gamma}\rho \lVert  u-\widetilde{u}\rVert^2_{L^2(K)} \;d\xi + \frac{c_{cv}}{8} \int \limits_{\Gamma}\rho \lVert \chi\rVert^2_e \;d\xi \\
	 &\leq& C\Bigg(h^{\min(\ell+1,2)-1} \lVert u \rVert_{L^2(H^2(\mathcal{D});\Gamma)} + \sum_{n=1}^{N}\bigg(\frac{k_n}{2}\bigg)^{q_n+1}\dfrac{\lVert \partial^{q_n+1}_{\xi_n} u \rVert_{L^2(H^1(\mathcal{D});\Gamma)} }{(q_n+1)!}\Bigg)^2 \\
 &&+ \frac{c_{cv}}{8}\lVert \chi\rVert_{\xi}^2, \\
	\lvert T_6\rvert \hspace{-3mm}&\leq&\hspace{-3mm} \int \limits_{\Gamma}\rho  \Bigg( \sum \limits_{ E \in \mathcal{E}^{0}_h}  \lVert \sqrt{\mathbf {b} \cdot \mathbf{n_E}}\big( u-\widetilde{u}\big)\rVert^2_{L^2(E)} \Bigg)^{\frac{1}{2}}\Bigg( \sum \limits_{E \in \mathcal{E}^{0}_h } \lVert \sqrt{\mathbf {b} \cdot \mathbf{n_E}}  \big(\chi^e -\chi\big)\rVert_E^2 \Bigg)^{\frac{1}{2}} d\xi \\
	&\leq& \hspace{-3mm}\frac{2}{c_{cv}}C \hspace{-2mm}\sum \limits_{ K \in \mathcal{T}_h} \int \limits_{\Gamma}\hspace{-1.5mm}\rho \bigg( \lVert u-\widetilde{u} \rVert_{L^2(K)} + h_{K} \lVert \nabla(u-\widetilde{u}) \rVert_{L^2(K)} \bigg)^2 \;\hspace{-2mm}d\xi + \frac{c_{cv}}{8} \int \limits_{\Gamma}\hspace{-1.5mm}\rho \lVert \chi\rVert^2_e \;d\xi\\
	&\leq& \hspace{-3mm}C\Bigg(h^{\min(\ell+1,2)-1} \lVert u \rVert_{L^2(H^2(\mathcal{D});\Gamma)} + \sum_{n=1}^{N}\bigg(\frac{k_n}{2}\bigg)^{q_n+1}\dfrac{\lVert \partial^{q_n+1}_{\xi_n} u \rVert_{L^2(H^1(\mathcal{D});\Gamma)} }{(q_n+1)!}\Bigg)^2 \\
&&+ \frac{c_{cv}}{8}\lVert \chi\rVert_{\xi}^2,
\end{eqnarray*}
\begin{eqnarray*}
	\lvert T_7\rvert \hspace{-3mm}&\leq&\hspace{-3mm}\int \limits_{\Gamma}\rho  \Bigg( \sum \limits_{ E \in \mathcal{E}^{0}_h}  \lVert \sqrt{\mathbf {b} \cdot \mathbf{n_E}}\big(  u-\widetilde{u}\big)\rVert^2_{L^2(E)} \Bigg)^{\frac{1}{2}}\Bigg( \sum \limits_{E \in \mathcal{E}^{\partial}_h } \lVert \sqrt{\mathbf {b} \cdot \mathbf{n_E}} \chi \rVert_E^2 \Bigg)^{\frac{1}{2}} d\xi \\
	&\leq& \hspace{-3mm}\frac{2}{c_{cv}}C \hspace{-2mm}\sum \limits_{ K \in \mathcal{T}_h} \int \limits_{\Gamma}\hspace{-1.5mm}\rho \bigg( \lVert u-\widetilde{u} \rVert_{L^2(K)} + h_{K} \lVert \nabla(u-\widetilde{u}) \rVert_{L^2(K)} \bigg)^2 \;\hspace{-2mm}d\xi + \frac{c_{cv}}{8} \int \limits_{\Gamma}\hspace{-1.5mm}\rho \lVert \chi\rVert^2_e \;d\xi\\
	&\leq& \hspace{-3mm}C\Bigg(h^{\min(\ell+1,2)-1} \lVert u \rVert_{L^2(H^2(\mathcal{D});\Gamma)} + \sum_{n=1}^{N}\bigg(\frac{k_n}{2}\bigg)^{q_n+1}\dfrac{\lVert \partial^{q_n+1}_{\xi_n} u \rVert_{L^2(H^1(\mathcal{D});\Gamma)} }{(q_n+1)!}\Bigg)^2 \\
&&+ \frac{c_{cv}}{8}\lVert \chi\rVert_{\xi}^2.
\end{eqnarray*}
Combining the bounds of $T_1$--$T_7$, we obtain the following result
\begin{eqnarray}\label{errm1}
\lVert u_h-\widetilde{u}\rVert_{\xi}  &\leq& C \left( h^{\min(\ell+1,2)-1} \lVert u \rVert_{L^2(H^2(\mathcal{D});\Gamma)}   \right. \\
&& \quad +\sum_{n=1}^{N}\bigg(\frac{k_n}{2}\bigg)^{q_n+1}\dfrac{\lVert \partial^{q_n+1}_{\xi_n} u \rVert_{L^2(H^1(\mathcal{D});\Gamma)} }{(q_n+1)!}\Bigg). \nonumber
\end{eqnarray}

Now, we discuss the second term in \eqref{err3}, i.e., $\lVert u-\tilde{u}\rVert_{\xi}$. By the definition of energy norm in \eqref{energynorm}, we have
\begin{eqnarray*}
   \lVert u-\widetilde{u}\rVert_{\xi}^2 &=&  \int \limits_{\Gamma}\rho  \lVert u-\widetilde{u}\rVert_e^2 \;d\xi \\
   &=& \int \limits_{\Gamma}\rho  \Bigg[
   \sum \limits_{K \in \mathcal{T}_h} \int \limits_{K} a(.,\omega) (\nabla(u-\widetilde{u}))^2\; dx +\hspace{-2mm} \sum \limits_{ E \in \mathcal{E}^{0}_h \cup \mathcal{E}_h^{\partial}} \frac{\sigma}{h_E} \int \limits_E \jump{u-\widetilde{u}}^2 \; ds \\
  && + \frac{1}{2}\sum \limits_{ E \in \mathcal{E}_h^{\partial}}\int \limits_E \mathbf {b}(.,\omega)\cdot \mathbf{n_E}(u-\widetilde{u})^2 ds \\
  && + \frac{1}{2}\sum \limits_{ E \in \mathcal{E}^{0}_h }\int \limits_E \mathbf {b}(.,\omega)\cdot \mathbf{n_E}((u-\widetilde{u})^e-(u-\widetilde{u}))^2 ds
    \Bigg]  d\xi\\
    &=& A_1+A_2+A_3+A_4.
\end{eqnarray*}
One can easily derive the following estimates as done in previous steps
\begin{eqnarray*}
	A_1	\hspace{-3mm}& \leq & \hspace{-3mm} \int \limits_{\Gamma}\rho  a_{\max}	\sum \limits_{K \in \mathcal{T}_h} \lVert \nabla(u-\widetilde{u}) \rVert^2_{L^2(K)} \; d\xi \\
	&=& \hspace{-3mm}C\lVert  \nabla(u-\widetilde{u}) \rVert^2_{L^2(L^2(\mathcal{D});\Gamma)} \\
	&\leq&\hspace{-3mm} C\Bigg(h^{\min(\ell+1,2)-1} \lVert u \rVert_{L^2(H^2(\mathcal{D});\Gamma)} + \sum_{n=1}^{N}\bigg(\frac{k_n}{2}\bigg)^{q_n+1}\dfrac{\lVert \partial^{q_n+1}_{\xi_n} u \rVert_{L^2(H^1(\mathcal{D});\Gamma)} }{(q_n+1)!}\Bigg)^2,
\end{eqnarray*}
\begin{eqnarray*}
	A_2	\hspace{-3mm}&\leq&\hspace{-3mm} \int \limits_{\Gamma}\rho \sum \limits_{ E \in \mathcal{E}^{0}_h \cup \mathcal{E}_h^{\partial}} \frac{\sigma }{h_E} \Bigg( Ch_E^{\frac{1}{2}} \lvert K_1^E\rvert^{-\frac{1}{2}} \big( \lVert u-\widetilde{u} \rVert_{L^2(K_1^E)} + h_{K_1^E}\lVert \nabla(u-\widetilde{u}) \rVert_{L^2(K_1^E)}\big)   \\
	&& \hspace{-3mm}+ Ch_E^{\frac{1}{2}} \lvert K_2^E\rvert^{-\frac{1}{2}} \big( \lVert u-\widetilde{u} \rVert_{L^2(K_2^E)} + h_{K_2^E}\lVert \nabla(u-\widetilde{u}) \rVert_{L^2(K_2^E)}\big) \Bigg)^2\\
	&\leq& \hspace{-3mm} \int \limits_{\Gamma}\rho  \sum \limits_{K \in \mathcal{T}_h}   C\bigg( \lVert u-\widetilde{u} \rVert_{L^2(K)} + h_{K} \lVert \nabla(u-\widetilde{u}) \rVert_{L^2(K)} \bigg)^2  d\xi\\
	&\leq& \hspace{-3mm}C \bigg( \lVert u-\widetilde{u} \rVert_{L^2(L^2(\mathcal{D});\Gamma)} + h \lVert \nabla(u-\widetilde{u}) \rVert_{L^2(L^2(\mathcal{D});\Gamma)} \bigg)^2\\
	&\leq&\hspace{-3mm} C\Bigg(h^{\min(\ell+1,2)-1} \lVert u \rVert_{L^2(H^2(\mathcal{D});\Gamma)} + \sum_{n=1}^{N}\bigg(\frac{k_n}{2}\bigg)^{q_n+1}\dfrac{\lVert \partial^{q_n+1}_{\xi_n} u \rVert_{L^2(H^1(\mathcal{D});\Gamma)} }{(q_n+1)!}\Bigg)^2, \\
	A_3	\hspace{-3mm}&\leq&\hspace{-3mm} \frac{1}{2}\int \limits_{\Gamma}\rho \sum \limits_{ E \in \mathcal{E}_h^{\partial}} \vert \mathbf{b}\cdot \mathbf{n_E}\rvert \lVert u-\widetilde{u} \rVert_{L^2(E)}^2\;d\xi\\
	&\leq&\hspace{-3mm} C\int \limits_{\Gamma}\rho\big( \lVert u-\widetilde{u} \rVert_{L^2(\mathcal{D})} + h\lVert \nabla(u-\widetilde{u}) \rVert_{L^2(L^2(\mathcal{D});\Gamma)}\big)^2\;d\xi\\
	&\leq&\hspace{-3mm} C\Bigg(h^{\min(\ell+1,2)-1} \lVert u \rVert_{L^2(H^2(\mathcal{D});\Gamma)} + \sum_{n=1}^{N}\bigg(\frac{k_n}{2}\bigg)^{q_n+1}\dfrac{\lVert \partial^{q_n+1}_{\xi_n} u \rVert_{L^2(H^1(\mathcal{D});\Gamma)} }{(q_n+1)!}\Bigg)^2, \\
	A_4\hspace{-3mm}	&\leq& \hspace{-3mm} \frac{1}{2}\int \limits_{\Gamma}\rho \sum \limits_{ E \in \mathcal{E}^{0}_h } \lvert \mathbf{b} \cdot \mathbf{n_E} \rvert \bigg(\lVert (u-\widetilde{u})^e\rVert_{L^2(E)}+ \lVert (u-\widetilde{u}) \rVert_{L^2(E)} \bigg)^2 \;d\xi \\
	&\leq&\hspace{-3mm} C\big( \lVert u-\widetilde{u} \rVert_{L^2(L^2(\mathcal{D});\Gamma)} + h\lVert \nabla(u-\widetilde{u}) \rVert_{L^2(L^2(\mathcal{D});\Gamma)}\big)^2\;\\
	&\leq&\hspace{-3mm} C\Bigg(h^{\min(\ell+1,2)-1} \lVert u \rVert_{L^2(H^2(\mathcal{D});\Gamma)} + \sum_{n=1}^{N}\bigg(\frac{k_n}{2}\bigg)^{q_n+1}\dfrac{\lVert \partial^{q_n+1}_{\xi_n} u \rVert_{L^2(H^1(\mathcal{D});\Gamma)} }{(q_n+1)!}\Bigg)^2.
\end{eqnarray*}
Summation of the bounds of $A_1$--$A_4$ gives us
\begin{eqnarray}\label{errm2}
\lVert u-\widetilde{u}\rVert_{\xi}^2 &\leq& C\Bigg(h^{\min(\ell+1,2)-1} \lVert u \rVert_{L^2(H^2(\mathcal{D});\Gamma)} \nonumber  \\
&& \qquad + \sum_{n=1}^{N}\bigg(\frac{k_n}{2}\bigg)^{q_n+1}\dfrac{\lVert \partial^{q_n+1}_{\xi_n} u \rVert_{L^2(H^1(\mathcal{D});\Gamma)} }{(q_n+1)!}\Bigg)^2 .
\end{eqnarray}
Finally, we obtain the desired result from \eqref{errm1} and \eqref{errm2}.
\end{document}